\documentclass[twoside,12pt]{amsart}
\usepackage{amscd, bbm, amsaddr}
\usepackage{amssymb,mathabx, eucal}
\usepackage{mathrsfs,wasysym, tikz-cd, extpfeil}
\usetikzlibrary{matrix, arrows}
\usepackage{amsmath}
\usepackage[margin=2cm]{geometry}

\usepackage{color}
\usepackage{amscd}
\usepackage[pdftex, colorlinks, linkcolor = blue, urlcolor = blue, citecolor = blue]{hyperref}

\makeatletter
\@addtoreset{equation}{section}

\usepackage{setspace}
\numberwithin{equation}{section}

\newtheorem{Theorem}{Theorem}[section]
\newtheorem*{Theorem*}{Theorem}
\newtheorem*{Corollary*}{Corollary}

\newtheorem{Lemma}[Theorem]{Lemma}
\newtheorem{Proposition}[Theorem]{Proposition}
\newtheorem{Corollary}[Theorem]{Corollary}

\newtheorem*{Conjecture}{Conjecture}

\theoremstyle{definition}

\theoremstyle{remark}
\newtheorem{Remark}[Theorem]{Remark}
\newtheorem*{Remark*}{Remark}

\newbox\squ  
\setbox\squ=\hbox{\vrule width.6pt
   \vbox{\hrule height.6pt width.4em\kern1ex
         \hrule height.6pt}%
   \vrule width.6pt}

\newcommand{\C}{\mathbb{C}}

\newcommand{\Z}{\mathbb{Z}}
\renewcommand{\k}{\mathbbm{k}}
\newcommand{\Fbb}{\mathbb{F}}
\renewcommand{\Z}{\mathbb{Z}}

\newcommand{\g}{\mathfrak{g}}
\newcommand{\gl}{\mathfrak{gl}}
\renewcommand{\sl}{\mathfrak{sl}}
\newcommand{\so}{\mathfrak{so}}

\newcommand{\p}{\mathfrak{p}}
\newcommand{\h}{\mathfrak{h}}

\newcommand{\n}{\mathfrak{n}}
\newcommand{\m}{\mathfrak{m}}
\renewcommand{\l}{\mathfrak{l}}
\newcommand{\z}{\mathfrak{z}}
\renewcommand{\r}{\mathfrak{r}}
\newcommand{\kk}{\mathfrak{k}}
\renewcommand{\b}{\mathfrak{b}}
\newcommand{\s}{\mathfrak{s}}
\renewcommand{\v}{\mathfrak{v}}

\renewcommand{\O}{\mathcal{O}}
\newcommand{\N}{\mathcal{N}}
\newcommand{\Ss}{\mathcal{S}}

\newcommand{\V}{\mathcal{V}}

\newcommand{\ad}{\operatorname{ad}}

\newcommand{\Lie}{\operatorname{Lie}}

\newcommand{\rank}{\operatorname{rank}}
\newcommand{\Char}{\operatorname{char}}
\renewcommand{\span}{\operatorname{span}}
\newcommand{\reg}{{\operatorname{reg}}}
\newcommand{\Hom}{\operatorname{Hom}}
\newcommand{\Spec}{\operatorname{Spec}}

\newcommand{\Mat}{\operatorname{Mat}}
\newcommand{\Aut}{\operatorname{Aut}}

\newcommand{\Ind}{\operatorname{Ind}}

\newcommand{\lmod}{\operatorname{-mod}}

\newcommand{\GL}{\operatorname{GL}}

\newcommand{\Tr}{\operatorname{Tr}}

\newcommand{\ind}{\operatorname{ind}}

\newcommand{\Diff}{\operatorname{Diff}}

\newcommand{\Ext}{\operatorname{Ext}}

\newcommand{\proj}{\operatorname{proj}}
\newcommand{\kalg}{\k\operatorname{-alg}}

\newcommand{\bkappa}{\hat\kappa}

\newcommand{\isoto}{\overset{\sim}{\longrightarrow}}
\newcommand{\onto}{\twoheadrightarrow}
\newcommand{\into}{\hookrightarrow}

\mathchardef\mh="2D

\title[]{\boldmath Modular representations of truncated current Lie algebras}

\begin{document}
\author{Matthew Chaffe and Lewis Topley}

\maketitle


\begin{abstract}
In this paper we consider the structure and representation theory of truncated current algebras $\g_m = \g[t]/(t^{m+1})$ associated to the Lie algebra $\g$ of a standard reductive group over a field of positive characteristic. We classify semisimple and nilpotent elements and describe their associated support varieties. Next, we prove various Morita equivalences for reduced enveloping algebras, including a reduction to nilpotent $p$-characters, analogous to a famous theorem of Friedlander--Parshall. 

We go on to give precise upper bounds for the dimensions of simple modules for all $p$-characters, and give lower bounds on these dimensions for homogeneous $p$-characters. We then develop the theory of baby Verma modules for homogeneous $p$-characters and, whenever the $p$-character has standard Levi type, we give a full classification of the simple modules. In particular we classify all simple modules with homogeneous $p$-characters for $\g_m$ when $\g = \gl_n$. Finally, we compute the Cartan invariants for the restricted enveloping algebra $U_0(\g_m)$ and show that they can be described by precise formulae depending on decomposition numbers for $U_0(\g)$.
\end{abstract}

\newcounter{parno}
\renewcommand{\theparno}{\thesection.\arabic{parno}}
\newcommand{\parn}{\refstepcounter{parno}\noindent\textbf{\theparno .} \space} 

\section{Introduction}

This story begins with the early work of Takiff \cite{Ta}, who showed that if $\g$ is a complex semisimple Lie algebra then the $\g[t]/(t^2)$ admits a polynomial ring of symmetric invariants, and went on to describe precise generators. This was later generalised by Ra{\"i}s--Tauvel \cite{RT} who proved a similar result for the {\it truncated current Lie algebras} $\g_m := \g[t]/(t^{m+1})$. \vspace{4pt}

In the case $m = 0$ the polynomiality of the symmetric invariants is a classical theorem of Chevalley, and one of the most notable consequences was Harish--Chandra's description of the centre of the enveloping algebra $U(\g)$. This was a key tool in the famous work of Bernstein--Gelfand--Gelfand which introduced the category $\O \subseteq \g\lmod$. In a recent development \cite{MS, Ch, CT} the work of Takiff and Ra{\"i}s--Tauvel was a fundamental tool in the development of the category $\O$ for truncated current algebras in characteristic zero. Our main results allowed us to express the composition multiplicities of simple modules in Verma modules, in terms of Kazhdan--Lusztig polynomials. The latter are known to control the decomposition multiplicities when $m = 0$ by the localisation theorem of Beilinson--Bernstein \cite{BB}. \vspace{4pt}

The modular representation theory of Lie algebras beyond the reductive case remains relatively unexplored, although it is expected that certain key theorems will generalise to all restricted Lie algebras; see \cite{KW, JaSupp, Na, Kac, PrComp, PSk, MST} for example. The representation theory of truncated current Lie algebras is a perfect testing ground for such generalisations, since they share many of the basic structural features of reductive Lie algebras. \vspace{4pt}

Now let $\k$ be an algebraically closed field of positive characteristic, let $G$ be an algebraic group over $\k$, and let $\g = \Lie(G)$. The structure of the category of representations of $\g$ is much more mysterious than the complex setting, and exhibits fundamentally different structural features. For example, all simple modules are finite dimensional, due to the existence of a large central subalgebra known as the $p$-centre $Z_p(\g)$. Kac--Weisfeiler \cite{KW} made the first headway in this theory by identifying $\Spec Z_p(\g)$ with the Frobenius twist of $\g^*$ as $G$-schemes, which leads to the consideration of the reduced enveloping algebras $U_\chi(\g)$ for $\chi\in \g^*$. These algebras are all isomorphic over a given coadjoint orbit, and the collection of their simple modules are precisely the simple $\g$-modules. Kac--Weisfeiler went on to classify simple $\g$-modules when $\g$ is completely solvable, and they presented strong evidence for two famous conjectures which relate the representation theory of $U_\chi(\g)$ to the geometry of the coadjoint orbit $G\cdot \chi$. These conjectures, which should be viewed as an expression of Kirillov's orbit method, eventually became known as KW1 and KW2 (see Section~\ref{S:KWconj} for a detailed discussion).\vspace{4pt}

Now let $G$ be reductive and assume the standard hypotheses (see \cite[\textsection 6.3]{JaLA}); we call such a group a {\it standard reductive group}. Many authors contributed to the study of the symmetric invariants $S(\g)^\g$ and the centre $Z(\g)$, but the first and most significant development was made by Veldkamp \cite{Ve}, who showed that the symmetric invariant algebra $S(\g)^\g$ is generated by the $p$th powers $S(\g)^p$ along with $S(\g)^G$. The latter is described by Chevalley restriction, as in characteristic zero. Similarly, he showed that the centre of the enveloping algebra is generated by $Z_p(\g)$ and $U(\g)^G$.\vspace{4pt}

Another breakthrough in the study of $\g$-modules is the work of Friedlander--Parshall \cite[Theorem~3.2]{FPmod} where they introduced the theory of support varieties for reduced enveloping algebras, and used this to reduce the study of $U_\chi(\g)$-modules to the case where $\chi$ is associated to a nilpotent element under some choice of $G$-equivariant isomorphism from $\g$ to $\g^*$. They also introduced the notion of nilpotent elements of {\it standard Levi type} and demonstrated that simple modules could be classified for nilpotent $p$-characters of this type.\vspace{4pt}

Later Premet \cite{PrIR} proved the KW2 conjecture for $G$ reductive using the theory of support varieties and a delicate construction adapted to each nilpotent orbit (see \cite[\textsection 8]{JaLA} for a short survey). Premet provided two subsequent proofs of this result, the first in collaboration with Skryabin \cite[Theorem~5.6]{PSk} and the second whilst introducing finite $W$-algebras \cite{PrST}. \vspace{4pt}

The purpose of the present article is to develop a theory of the representations of truncated current Lie algebras $\g_m$ over $\k$ in the case where $\g$ is the Lie algebra of a standard reductive group. In certain cases we are able to give a complete generalisation of some of the results noted above, whilst in other cases our results are require additional hypotheses (void in the case $m=0$).\vspace{4pt}

We now proceed to describe the structure of the paper and highlight some of our notable results. Fix $m > 0$.\vspace{6pt}
\begin{enumerate}
\setlength{\itemsep}{6pt}
\item[(\textsection\ref{S:prelim})] We recall some of the basic theory of restricted Lie algebras and their support varieties. Then we introduce standard properties of a reductive group scheme $G$ over $\k$, following \cite{JaRAGs, Mil}. We recap the standard hypotheses of $\g = \Lie(G)$ from \cite{JaLA}. Finally we introduce the theory of jet schemes of groups following \cite{Ish}: these play a key role, because the $\k$-points $G_m := J_m G(\k)$ form an algebraic group, known as the {\it truncated current group}. This group acts by automorphisms on $\g_m$, and these automorphisms play an analogous role to the adjoint action of $G$ when studying representations of $\g$. We finish the section by proving that $S(\g_m) \isoto U(\g_m)$ as $\g_m$-modules (Proposition~\ref{P:Milnersmapandsmoothness}) and describing the support variety of the category of $U_\chi(\g_m)$-modules for each $\chi \in \g^*$ (Corollary~\ref{C:supportvarieties}).

\item[(\textsection \ref{S:truncatedcurrentalgs})] Our results here are structural: we show that every semisimple element of $\g_m$ is $G_m$-conjugate to an element of graded degree zero (Proposition~\ref{P:semisimplesincurrents}) and classify the nilpotent elements with respect to the natural restricted structure (Proposition~\ref{P:nilpotentsincurrents}). Next we go on to show that the adjoint representation admits generic adjoint stabilisers under $G_m$, and use this to calculate the index of $\g_m$ (Proposition~\ref{P:genericstabiliser}). We also make some remarks on the regular elements of $\g_m$ and compare to the characteristic zero case dealt with in \cite[Theorem~2.8]{RT}. Finally, we recall the structure of the symmetric invariants of $\g_m$ given in \cite{ATV}, and use this to give a description the centre of $U(\g_m)$, which is a fusion of the main results of \cite{Ve} and \cite{RT}.

\item[(\textsection \ref{S:repthy})] In this section we begin our investigation of the representation theory of $\g_m$, starting by proving three category equivalences. The first allows us to reduce the study of simple $U_\chi(\g_m)$-modules to the case where $\chi$ is non-vanishing on $\g_m^{(m)} := \g \otimes t^m \subseteq \g_m$ (Proposition~\ref{P:reductionbydegree}). The second provides equivalences between $U_\chi(\g)\lmod$ and $U_{\chi'}(\g_m)\lmod$ whenever $\chi - \chi'$ vanishes on the derived subalgebra (Lemma~\ref{L:reducingoverthecentre}). The third equivalence is a precise analogue of a famous theorem of Friedlander--Parshall \cite[Theorem~3.2]{FPmod} which allows us to reduce the study of simple $\g$-modules to the case where the $p$-character is nilpotent (Theorem~\ref{T:parabolicinduction}). It is worth mentioning that an analogous result was proven by the current authors over $\C$, inspired by the theory of modular Lie algebras \cite[Theorem~4.2]{CT}. Finally we apply our previous results to give a precise description of the simple and projective $U_\chi(\g_m)$-modules when $\chi$ is regular semisimple (Corollary~\ref{C:semisimplecharactermodules} and Proposition~\ref{P:torusprojectives}). 

\item[(\textsection \ref{S:KWconj})] In this section we study the Kac--Weisfeiler conjectures and the generalisation of KW2 proposed in \cite{Kac}, for $\g_m$. We begin by recounting the known partial results, and then we prove KW1 for $\g_m$ (Corollary~\ref{C:KW1forgm}). We also prove KW2 for $\g_m$ when $\g = \sl_2$ (Proposition~\ref{P:KW2forsl2}), and provide some partial results towards the general case (Proposition~\ref{P:KW2pcharacters}).

\item[(\textsection \ref{S:classifyingsimples})] Here we turn our attention to the classification of simple $\g_m$-modules. Thanks to our various Morita equivalences, it suffices to focus on nilpotent $p$-characters. The general classification problem is out of reach at present, but we are able to give a complete classification under the assumption that the nilpotent $p$-character is homogeneous in some degree (i.e. vanishes on $\bigoplus_{i\ne k}\g_m t^i$ for some $k$) and has standard Levi type. Under these assumptions we develop a theory of baby Verma modules, show that every baby Verma has a unique simple quotient, and show that these simple modules exhaust all the simple $U_\chi(\g_m)$-modules. We give a necessary and sufficient condition for two simple quotients of baby Verma modules to be isomorphic (see \S\ref{ss:simpleclassification}).  The main result of this section is presented in \S\ref{ss:generallinear}, where we explain how to classify the simple $U_\chi(\g_m)$-modules when $\chi \in \g_m^*$ is homogeneous and $\g = \gl_n$.

\item[(\textsection \ref{S:Cartaninv})] The final section studies the Cartan invariants of the restricted enveloping algebra $U_0(\g_m)$. Recall that these are the composition multiplicities of simple modules in the indecomposable projectives. One of the key ingredients is the theory of characters, which we introduce using the technique (well-known from the case $m = 0$) of studying $U_0(\g)$-modules graded by a maximal torus of $G$. Another important ingredient is a very general theorem of Nakano \cite[Theorem~1.3.6]{Na}. We give a complete description of the blocks of $U_0(\g_m)$ (see Corollary~\ref{C:blocks} and Remark~\ref{R:blocks}). Our final result is Theorem~\ref{T:restrictedprojectivescompmults} which determines the Cartan invariants of $U_0(\g_m)$ by precise formulae depending on the decomposition multiplicities for $U_0(\g)$. In Remark~\ref{R:final} we observe that these formulae can essentially be computed in terms of $p$-Kazhdan--Lusztig polynomials provided $p$ is greater than $2h-1$, where $h$ denotes the Coxeter number, thanks to the recent work of Riche--Williamson \cite{RW}.\vspace{6pt}
\end{enumerate}

We conclude the introduction by mentioning another milestone achievement in the theory of $\g$-modules: the proof by Bezrukavnikov--Mirkovic--Rumynin \cite{BMR} of a derived version of the localisation theorem of \cite{BB}. This work begins with a careful study of the geometry of the centre $\Spec Z(\g)$, building on Veldkamp's theorem. An interesting challenge now is to develop a geometric description of $\g_m$-modules and we hope that our work will serve as the groundwork for such a theory.

\subsection*{Acknowledgements}
Both authors would like to thank Simon Goodwin and Matt Westaway for many useful conversations. Thanks also to Alexander Premet for comments on the initial draft of this article. The first author is grateful to the EPSRC for studentship funding, and the research of the second author is supported by a UKRI FLF, grant numbers MR/S032657/1, MR/S032657/2, MR/S032657/3.

\section{Preliminaries}
\label{S:prelim}

Throughout this paper we let $\k$ denote an algebraically closed field of characteristic $p > 0$. All vector spaces, algebras and schemes will be defined over $\k$. The Frobenius twist of a vector space is denoted $V^{(1)}$, and similarly $V^{(r)}$ denotes the vector space structure on $V$ with $\k$-action twisted by the $r$th power of the Frobenius map, i.e. $\lambda \cdot_r v = \lambda^{1/p^r} v$ for $\lambda\in \k$ and $v\in V$.

\subsection{Restricted Lie algebras, reduced enveloping algebras and Jordan decomposition} Let $\g$ be a finite-dimensional restricted Lie algebra with $p$-mapping $x \mapsto x^{[p]}$. This satisfies the property that the map $x \mapsto x^p - x^{[p]}$ is $p$-semilinear and has image contained in the centre $Z(\g)$ of the enveloping algbera $U(\g)$; see \cite[\textsection 2]{SF} for a good introduction to the basics. In particular, $U(\g)$ admits a large central subalgebra $Z_p(\g)$ generated by expressions of the form $x^p - x^{[p]}$. It follows from the axioms of a restricted structure that the map $\k[(\g^*)^{(1)}] = \k[\g^*]^p \to Z_p(\g)$ given by $x^p \mapsto x^p - x^{[p]}$, for $x\in \g = (\g^*)^*$, is an isomorphism of $\k$-algebras, and so it is conventional to identify $\Spec Z_p(\g) = (\g^*)^{(1)}$ (an observation of \cite{KW}). Since $\k$ is perfect there is a natural identification $\g^* = (\g^*)^{(1)}$ of topological spaces and so the maximal spectrum of $Z_p(\g)$ is naturally labelled by $\g^*$. We write $I_\chi$ for the ideal corresponding to $\chi \in \g^*$.

If $G$ is an algebraic group scheme over $\k$ then the Lie algebra $\g = \Lie(G)$ acquires a natural $G$-equivariant restricted structure. The adjoint action of $G$ on $\g$ extends to $U(\g)$ and preserves the $p$-centre. The identification $\Spec Z_p(\g) = \g^*$ is also $G$-equivariant (see \cite[4.4.3]{Spr} for more detail).

The {\it reduced enveloping algebra with $p$-character $\chi$} is the algebra $U_\chi(\g) := U(\g) / I_\chi U(\g)$. The significance of these algebras arises from the fact that every irreducible representation of $\g$ factors through precisely one of the quotients $U(\g) \onto U_\chi(\g)$. The most important member of this family of algebras is the {\it restricted enveloping algebra $U_0(\g)$}.

For $x\in \g$ we write $\g^x$ for the adjoint centraliser. We recall that an element $x\in \g$ is called {\it semisimple} if $x$ lies in the span of $\{x^{[p]^i} \mid i > 0\}$ and an element is {\it nilpotent} if $x^{[p]^i} = 0$ for $i \gg 0$.
The set of nilpotent elements is denoted $\N(\g)$ and called the {\it nilpotent cone}. The subset of elements satisfying $x^{[p]} = 0$ is the $p$-nilpotent cone, denoted $\N^{[p]}(\g)$. A restricted Lie algebra $\g$ is said to be {\it unipotent} if $\g^{[p]^i} = 0$ for $i \gg 0$.

Although the following observation is well-known, we could not find the final statement in the literature, and so we provide a short proof for completeness.
\begin{Lemma}
\label{L:Jordandecomposition}
For $x \in \g$ there exist unique decomposition $x = s + n$ into commuting semisimple and nilpotent parts. Furthermore, $\g^x = (\g^s)^n$.
\end{Lemma}
\begin{proof}
The existence and uniqueness of such elements is \cite[Theorem~2.3.5]{SF}. Recall that $\g$ admits a faithful restricted representation $\g \into \gl(V)$, for example by taking $V = U_0(\g)$ under left multiplication. By uniqueness the Jordan decomposition of $x$ in $\gl(V)$ coincides with that in $\g$. In $\gl(V)$ the equality $\gl(V)^x = (\gl(V)^s)^n$ follows from the classical Jordan decomposition theorem, and so the final remark of the proof is immediate, intersecting $\gl(V)^x$ with $\g$.
\end{proof}

\subsection{Support varieties}

If $\g$ is a restricted Lie algebra then the structure of the category of $U_\chi(\g)$-modules depends heavily on $\chi \in \g^*$. This dependence is illustrated by the theory of support varieties. If $\k$ is equipped with the trivial $\g$-module structure then the vector space $H(\g) := \bigoplus_{i\ge 0}\Ext_{U_0(\g)}^{2i}(\k, \k)^{(-1)}$ is equipped with a commutative associative multiplication via the Yoneda product. There is a morphism of algebraic varieties $\Phi : \Spec H(\g) \to \g$ known as the Hochschild map. By the work of Jantzen \cite{JaSupp}, it is finite and the image is the $p$-nilpotent cone $\N^{[p]}(\g) := \{ x\in \g \mid x^{[p]} = 0\}$.

If $M \in U_\chi(\g)\lmod$ then $H(\g)$ acts on the graded vector space $\Ext_{U_\chi(\g)}^\bullet (M, M)^{(-1)}$, thanks to \cite[Proposition~5.1]{FPmod}. The support variety $\V_\g(M)$ of $M$ is defined to be the image under $\Phi$ of the support of this $H(\g)$-module. By the above, it is a closed conical subvariety of $\N^{[p]}(\g)$.

We also define $\V_\g(\chi)$ to be the support variety of the category of (finite-dimensional) $U_\chi(\g)$-modules; that is, $\V_\g(\chi) = \V_\g(M_1 \oplus \dots \oplus M_k)$ where $M_1, \dots, M_k$ is a complete set of representatives of isoclasses of simple $U_\chi(\g)$-modules. It follows from \cite[Theorem~2.1(v), (vii)]{PrComp} that $\V_{\g}(M) \subseteq \V_{\g}(\chi)$ for any $U_\chi(\g)$ module $M$. 

One compelling feature of $\V_\g(M)$ is that the dimension is equal to the rate of polynomial growth of a minimal projective resolution of $M$ in the category of $U_\chi(\g)$-modules \cite[Remark~6.3]{FPmod}, and more detailed information is provided by the relation with rank varieties (Cf. Theorem~6.4 of {\it loc. cit.}). 

We refer the reader to Proposition~6.2 and Proposition~7.1 of {\it op. cit} for the proof of the following result.
\begin{Theorem}
\label{T:supportvarieties}
Let $M$ be a $U_\chi(\g)$-module. Then:
\begin{enumerate}
\item $M$ is a projective $U_
\chi(\g)$-module if and only if $\V_\g(M) = 0$.
\item If $\s \subseteq \g$ is a restricted Lie subalgebra then $\V_\s(M) = \V_\g(M) \cap \s$. $\hfill \qed$
\end{enumerate}
\end{Theorem}

\subsection{Standard group schemes}
\label{ss:groupschemes}
In this section we set up conventions regarding Lie algebras of group schemes over $\k$. Throughout the paper $G$ will be a reductive, connected $\k$-group scheme in the sense of \cite{JaRAGs, Mil}. Write $\g = \Lie(G)$.

We make the following assumptions, commonly known as the {\it standard hypotheses}.
\begin{enumerate}
\item[(H1)] The derived subgroup of $G$ is simply connected;
\item[(H2)] $p$ is good for the root system of $G$;
\item[(H3)] There exists a non-degenerate $G$-invariant bilinear form $\kappa$ on $\g$
\end{enumerate}
We call a group scheme satisfying these hypotheses a {\it standard reductive group scheme}. We note that when $G$ is standard reductive, so too are all of the Levi subgroups. Thus the hypotheses provide a framework which is amenable to inductive arguments.

It is stated without proof in \cite[\textsection 2.9]{JaNO}, and explained in detail in \cite[\textsection 2.1]{PSt}, that these hypotheses imply that the Lie algebra $\g$ is a product of Lie algebras of the following types:
\begin{itemize}
\item[(i)] simple Lie algebras not of type {\sf A},
\item[(ii)] $\sl_n$ with $p \nmid n$,
\item[(iii)] $\gl_n$ with $p \mid n$,
\item[(iv)] a restricted torus, i.e. an abelian Lie algebra with bijective $p$-mapping.
\end{itemize}

The next result follows from \cite[Theorem~A]{Gar} along with the previous observation and the fact that the natural representation of $\gl_n$ subtends a non-denegerate trace form for any $p > 0$. 
\begin{Lemma}
\label{L:faithfulrep}
If $G$ is standard reductive then there is a representation $\rho : \g \to \gl(V)$ for which the trace form $(x , y) \mapsto  \Tr(\rho(x) \rho(y))$ is $\ad(\g)$-invariant and non-degenerate. $\hfill\qed$
\end{Lemma}

Observe that $\kappa$ induces a $G$-equivariant isomorphism
\begin{eqnarray}
\label{e:kappaiso}
\bkappa : \g \isoto \g^*
\end{eqnarray}

Thanks to Lemma~\ref{L:faithfulrep} these group schemes satisfy {\it Richardson's property}: the representation $\g\to \gl(V)$ appearing in Lemma~\ref{L:faithfulrep} has the property that $\rho$ is faithful and there is an $\ad(\g)$-stable subspace $M\subseteq \gl(V)$ such that
\begin{eqnarray}
\label{e:Richardson}
\gl(V) = \g \oplus M
\end{eqnarray}
One may take $M$ to be the orthogonal complement to $\rho(\g)$ with respect to $\kappa$. We refer the reader to \cite[\textsection 2]{JaNO} for more detail on the consequences of this property, and \cite{Pr4} for a recent development, characterising Lie algebras admitting a module with Richardson's property when $p > 3$.

\subsection{Jet schemes and truncated current Lie algebras}

We refer the reader to \cite{Ish} for a nice general introduction to the theory of jet schemes, see also \cite{Mu} for more advanced results. Our exposition here is purely algebraic, and focuses on $m$-jets of group schemes over $\k$.

Let $A \subseteq B$ be commutative algebras. A {\it higher derivation $A \to B$ of order $m$}  is a tuple $\partial = (\partial^{(i)} \mid i = 0,...,m)$ of linear maps $A \to B$ with $\partial^{(0)}$ equal to the identity, satisfying the following relations for $i = 0,...,m$:
\begin{eqnarray}
\label{e:HSrel}
\begin{array}{l}
\partial^{(i)} (ab) = \sum_{j_1 + j_2 = i} \partial^{(j_1)}(a) \partial^{(j_2)}(b).
\end{array}
\end{eqnarray}
A differential algebra of order $m$ over $A$ is a pair $(B, \partial)$ where $B$ is a faithful $A$-algebra and $\partial : A \to B$ is as above. Morphisms between such algebras are the obvious commutative triangles (we assume that $B$ is a faithful $A$-algebra to keep our notation transparent). Write $\Diff_A^m$ for the category of differential algebras of order $m$ over $A$.

There is an endofunctor $J_m$ on the category of algebras such that for any $B \in \Diff_A^m$ we have
$$\Hom_{\kalg} (A, B) \cong \Hom_{\Diff_A^m}(J_m A, B).$$

The construction is quite straightforward: when $A = \k[x_1,...,x_n]/(f_1,...,f_q)$ we have $$J_m A \cong \k[\partial^{(i)} x_j \mid i=0,...,m, \ j=1,...,n] / (\partial^{(i)} f_j\mid i=0,...,m, \ j =1,...,q),$$
where the symbols $\partial^{(i)}f_j$ should be expressed in terms of the generators using \eqref{e:HSrel}. 

The functor sends Hopf algebras to Hopf algebras \cite[\textsection 2.1]{ATV} and if $G$ is an affine algebraic group scheme then the group scheme $J_m G := \Spec J_m \k[G]$ is known as {\it the scheme of $m$-jets on $G$}. We record here some of the key properties which will be required in the sequel.
\begin{Lemma}
\label{L:jetlemma}
The following hold:
\begin{enumerate}
\setlength{\itemsep}{2pt}
\item The group of $A$-points of $J_m G$ is isomorphic to the group of $A[t]/(t^{m+1})$-points of $G$, for any $\k$-algebra $A$.
\item $\g_m := \Lie J_m G \cong \g \otimes \k[t]/(t^{m+1})$.
\end{enumerate}
\end{Lemma}
\begin{proof}
Part (1) is usually taken as the definition of the functor $J_m$ \cite[Theorem~2.2]{Ish}. The fact that our construction of $J_m G$ satisfies the required property can be seen by comparing the proof of {\it loc. cit.} with \cite[\textsection 2.1]{ATV}. The Hopf structure on $J_m \k[G]$ is described explicitly in \cite[(2.2)]{ATV}.

Part (2) is well-known, see \cite[Appendix]{Mu} for example. 
\end{proof}

The group of $\k$-points of $J_m G$ is the {\it $m$th truncated current group}, and we denote it by $G_m$. Since $\k[t]/(t^{m+1})$ is finite dimensional, $G_m$ is an affine algebraic variety.

The Lie algebra $\g_m$ is called {\it the $m$th truncated current Lie algebra associated to $\g$}. Recall that the adjoint representation of every Lie algebra extends by derivations to an action on both the symmetric and enveloping algebras.

\begin{Proposition}
\label{P:Milnersmapandsmoothness}
There is an $\ad(\g_m)$-equivariant isomorphism $U(\g_m) \isoto S(\g_m)$.
\end{Proposition}
\begin{proof}
Thanks to \cite[Theorem~1.2]{FPrat} the claim is equivalent to the assertion that  the inclusion $\g_m \to U(\g_m)$ of $\ad(\g_m)$-modules splits. We first prove this splitting for $\g = \gl_n(\k)$, and then for any Lie algebra which is a product of Lie algebras of the form (i)--(iv) of Section~\ref{ss:groupschemes}. Throughout the proof we write $\k_m = \k[t]/(t^{m+1})$.

Set $\g = \gl_n(\k)$. Choosing a basis for $\k^n$ we can identify $\g_m = \gl_n(\k_m) = \Mat_n(\k_m)$. Write $\ast$ for the associative multiplication on $\Mat_n(\k_m)$. Under this identification there is an associative algebra homorphism $U(\gl_n(\k_m)) \to \Mat_n(\k_m)$ which sends PBW monomial $x_1\cdots x_n$ to the element $x_1\ast \cdots \ast x_n$, for $x_i \in \Mat_n(\k_m)$. This map gives the required splitting, and we have proven (1) for $\g = \gl_n$.

Now let $\g$ be a Lie algebra which is a product of algebras of the form (i)--(iv) of Section~\ref{ss:groupschemes}. According to \eqref{e:Richardson} there is a $G$-module $V = \k^n$ such and a subspace $M \subseteq \gl_n(\k)$ as per \eqref{e:Richardson} which is $\ad(\g)$-stable. We have a $\g_m$-equivariant decomposition $\gl_n(\k_n) = \g_m \oplus (M\otimes \k_m)$, thus we may define a $\g_m$-equivariant map $U(\g_m) \to U(\gl_n(\k_m)) \to \gl_n(\k_m) \to \g_m$. This map splits the inclusion $\g_m\to U(\g_m)$, and this completes the proof.
\end{proof}

The following result may be proven by following the proof of \cite[Proposition~6.3]{PSk}, which applies to any finite dimensional restricted Lie algebra admitting an isomorphism as per Proposition~\ref{P:Milnersmapandsmoothness}. We emphasise that $\g_m$ is regarded as a Lie $\k$-algebra, so that $\g_m^* = \Hom_\k(\g_m, \k)$.
\begin{Corollary}
\label{C:supportvarieties}
If $G$ is a standard reductive group then for every $\chi \in \g_m^*$ we have
$$\V_{\g_m}(\chi) = \N^{[p]}(\g_m^\chi).$$
\end{Corollary}

\section{Truncated current Lie algebras}
\label{S:truncatedcurrentalgs}

\subsection{Truncated current groups and algebras}
We keep fixed a standard reductive group scheme, as per the setup in Section~\ref{ss:groupschemes}. We now begin our study of the structure of truncated current Lie algebras $\g_m$.

Recall that $\g_m = \g \otimes \k[t]/(t^{m+1})$. We make the notation $x\otimes t^i = xt^i$, write $\g_m^{(i)} = \g t^i \subseteq \g_m$ and note that $\g_m = \bigoplus_{i=0}^m \g_m^{(i)}$ is a Lie algebra grading. We also use notation $\g_m^{(>i)}$ and $\g_m^{(\le i)}$ for the sum of graded subspaces, when $0 \le i \le m$. The $G_m$-equivariant restricted structure on $\g_m$ is determined by $xt^i = x^{[p]}t^{pi}$ for $x\in \g$ and $i \ge 0$. 

\begin{Lemma}
\label{L:groupstructure}
\begin{enumerate}
\item $G_1 \cong G \ltimes \g$ where $(1, \g)$ is an abelian normal subgroup, with $G$ acting on $(1,\g)$  via the adjoint representation. In particular the adjoint action of $(g,x)$ on $x_0 + x_1t \in \g_1$ is given by
$$(g, x) \cdot (x_0 + x_1t) = g\cdot x_0 + ([x, x_0] + g \cdot x_1)t.$$

\item For $0\le i \le m$ the Lie algebra $\g_m/ \g_m^{(>i)}$ has admits a natural action of $G_1$ by automorphisms, such that $(g, x) \in G_1$ acts on $\sum_{j=0}^i x_j t^j + \g_m^{(>i)}$ by
\begin{eqnarray}
\label{e:}
(g,x) \cdot (\sum_{j=0}^i x_j t^j + \g_m^{(>i)}) := \sum_{j=0}^{i-1} (g\cdot x_j)t^j + ([x, x_i] + g\cdot x_i)t^i + \g_m^{(>i)}
\end{eqnarray}
\end{enumerate}
\end{Lemma}
\begin{proof}
Note that $G_1$ is equal to the total space of the tangent bundle $TG$ over $G$, by definition. Furthermore this bundle can be trivialised: any choice of basis of $\g$ gives rise to a system of nowhere vanishing global sections $G \to TG$ by homogeneity. This proves part (1).

Now consider the unital subalgebra $A_i \subseteq \k[t]/(t^{m+1})$ generated by $t^i$. This inclusion of algebras gives rise to an embedding $G_m^i := G(A_i) \into G_m$. The action of $G_m^i$ on $\g_m / g_m^{(\ge i)}$ factors through the diagonal action of $G$, whilst the action of $G_m^i$ on $(\g + \g t^i)/\g_m^{(>i)}$ factors through the action of $G_1$ induced by the homomorphism $\k[t]/(t^2) \to \k[t]/(t^{i})$ where $t \mapsto t^i$. 
\end{proof}

Choose a maximal torus $T \subseteq G$ and let $\h = \Lie(T)$. Let $X^*(T)$ denote the character lattice of $T$ and $\Phi \subseteq X^*(T)$ the root system associated to $T$. Choose a system $\Phi^+ \subseteq \Phi$ of positive roots and let $\Delta = \{\alpha_1,...,\alpha_r\}$ denote the corresponding simple roots. We write $N := |\Phi^+|$ and note that 
\begin{eqnarray}
\label{e:rootequation}
\dim G = 2N + r
\end{eqnarray}
These choices give rise to a triangular decomposition $\g = \n^- \oplus\h\oplus \n^+$ of $\g$, and a corresponding decomposition $\g_m = \n^-_m \oplus \h_m \oplus \n^+_m$ of $\g_m$. Here $\h_m = \bigoplus_{i=0}^m \h t^i$ and similar for $\n_m^\pm$. Notice that for $h = \sum_{i=0} h_i t^i\in \h_m$ we have that $h_0$ is semisimple and $\sum_{i>0}h_i t^i$ is nilpotent, whilst $[h_0, h_i]= 0$ for all $i$. By Lemma~\ref{L:Jordandecomposition} we have:
\begin{Lemma}
\label{L:hmJdec}
The semisimple part in the Jordan decomposition of $h$ is $h_0$.
\end{Lemma}

There is a non-degenerate $G_m$-invariant symmetric associative form on $\g_m$ defined by
\begin{eqnarray}
\label{e:nondegenform}
\begin{array}{rcl}
\kappa_m & : & \g_m \times \g_m \longrightarrow \k,\\
& & xt^i , yt^j \longmapsto \delta_{i+j, m}\kappa(x, y).
\end{array}
\end{eqnarray}
where $\kappa: \g \times \g \rightarrow \k$ is the nondegenerate form from axiom (H3) of the standard hypotheses. Using the non-degenerate pairing  we can define an isomorphism similar to \eqref{e:kappaiso}
\begin{eqnarray}
\label{e:kappamiso}
\begin{array}{rcl}
\bkappa_m & : & \g_m \longrightarrow \g_m^*\\
& & x \longmapsto \kappa_m(x, \cdot).
\end{array}
\end{eqnarray}
Note that this isomorphism preserves the gradings on these modules: it sends homogeneous elements $x \in \g_m^{(i)}$ to linear functions which vanish on $\bigoplus_{j\ne m-i} \g_m^{(j)}$. We will call such elements of $\g^*$ {\it homogeneous}, with support in degree $m-i$.

\subsection{Semisimple and nilpotent elements of $\g_m$}
In this section we describe the nilpotent and semisimple elements of $\g_m$ in terms of those same elements in $\g$ with the help of the adjoint $G_m$-action. Recall that we identify $\g = \g_m^{(0)}$ as restricted Lie algebras.
\begin{Proposition}
\label{P:semisimplesincurrents}
Every semisimple element of $\g_m$ is conjugate to a semisimple element of $\g_m^{(0)}$.
\end{Proposition}

\begin{proof}
By \cite[15.3]{Hu}, since $\g_m$ is the Lie algebra of an algebraic group, any two Cartan subalgebras of $\g_m$ are conjugate under the adjoint action of $G_m$. Furthermore, observe that $\h_m$ is one such Cartan subalgebra; it is nilpotent, and since $\h$ is self-normalising in $\g$ we see that $\h_m$ is self-normalising in $\g_m$. Now, any semisimple element $s$ is contained in some maximal torus, whose centraliser is a Cartan subalgebra (see \cite[2.4.2]{SF}), and so $s$ is conjugate to an element $\sum_{i=0}^m h_i t^i$ of $\h_m$. The desired result then follows from Lemma~\ref{L:hmJdec}.
\end{proof}

\begin{Proposition}
\label{P:nilpotentsincurrents}
The nilpotent elements of $\g_m$ are precisely those of the form
\begin{eqnarray}
\label{e:nilpotentform}
x_0 + \sum_{i=1}^m x_it^i
\end{eqnarray}
where $x_0 \in \N(\g)$ is nilpotent and $x_1,...,x_m\in \g$.
\end{Proposition}
\begin{proof}
Let $x \in \g$ be of the form \eqref{e:nilpotentform}. Using the axioms of a restricted Lie algebra inductively, along with the fact that $\g_m^{(>0)}$ is a nilpotent ideal we see that $x^{[p]^i} = 0$ for large $i$, so $x$ is nilpotent.

Now suppose that $x$ is not of the form \eqref{e:nilpotentform} and assume for a contradiction that $x$ is nilpotent. Then writing $x = \sum_{i=0}^m x_i t^i$ we see that $x_0 = x_0^s + x_0^n$ is the Jordan decomposition of $x_0$ in $\g_m^{(0)} \cong \g$ with $x_0^s \ne 0$. According to \cite[Theorem~2.3.4]{SF} there exists a $k > 0$ such that $x^{[p]^k}$ is a semisimple element, hence $x^{[p]^k} = 0$ because $x$ is nilpotent. However since $\g_m^{(> 0)}$ is a nilpotent ideal stable under the $p$-mapping it follows from the axioms for restricted Lie algebras that if we write $x^{[p]^k} = \sum_{i=0}^m x_{i,k} t^i$ then $x_0^{[p]^k} = x_{0,k}$. However $x_0^{[p]^k} \ne 0$ for any $k > 0$. The contradiction completes the proof.
\end{proof}

Using the $G_m$-equivariant identification \eqref{e:kappamiso} we can transport the notion of semisimple and nilpotent elements to $\g^*_m$, and hence talk of {\it semisimple} and {\it nilpotent $p$-characters}. In particular, we observe that any $p$-character is conjugate to some $\chi$ with Jordan decomposition $\chi = \chi_s + \chi_n$ such that $\chi_s$ is supported on $\h_m^{(m)}$, meaning $\chi_s$ vanishes on $\h_m^{(<m)} \oplus \n^-_m \oplus \n^+_m$. 

\subsection{Index and regular elements}

If $V$ is a $\g_m$-module then {\it the index of $\g_m$ in $V$}, denoted $\ind(\g_m, V)$, is defined to be the minimal dimension of a stabiliser $\g_m^v$ as we vary $v\in V$. We define regular elements of $V$ to be those elements $v\in V$  for which $\dim \g_m^v = \ind(\g_m, V)$. It is well-known that the regular elements form a dense open subset of $V$.

As a special case we define {\it the index of $\g_m$} to be $\ind(\g_m) := \ind(\g_m, \g_m^*)$. This invariant was defined by Dixmier, and plays an important role in ordinary and modular representation theory.

We say that an algebraic group $K$ with $\kk = \Lie(K)$ admits {\it generic adjoint stabilisers} if there is a dense open subset $\O \subseteq \kk$ such that the centralisers $\kk^x$ are all $K$-conjugate over $\O$.

\begin{Proposition}
\label{P:genericstabiliser}
\begin{enumerate}
\item $G_m$ admits generic adjoint stabilisers.
\item $\ind(\g_m) = (m+1) \ind(\g)$.
\end{enumerate}

\end{Proposition}
\begin{proof}
Let $\h_m^\reg := \{ \sum x_it^i \in \h_m \mid x_0 \in \h^\reg\}$, where $\h^\reg$ denotes the elements of $\h$ which are regular in $\g$. For $x\in \h_m^\reg$ it is easy to check that $\g_m^x = \h_m$. Furthermore, $[\h, \g] \subseteq \n^- \oplus \n^+$ implies $[\h_m, \g_m] \subseteq \n_m^- \oplus \n_m^+$. By dimensional considerations we deduce that $\ad(x) \g_m = \n_m^- \oplus \n_m^+$ for $x \in \h_m^\reg$.

Consider the morphism
\begin{eqnarray}
\varphi : G_m \times \h_m \to \g_m
\end{eqnarray}
given by restricting the adjoint action $G_m \times \g_m \to \g_m$. As observed in \cite[Lemma~1.6]{TY} the differential $d_{(1,x)} \varphi : \g_m \times \h_m \to \g_m$ is given by
$$d_{(1,x)} \varphi (u,v) = [x, u] + v.$$
For $x\in \h_m^\reg$ we deduce that $d_{(1,x)} \varphi$ is surjective, and it follows quickly that $G_m \times \h_m^\reg \to \g_m$ is a dominant morphism. We have shown that the conjugates of $\h_m^\reg$ are dense in $\g_m$ and that a point of $\h_m^\reg$ has centraliser $\h_m$. This completes the proof of (1).

The set of regular elements of $\g_m$ is open dense and so the conjugates of $\h_m^{\reg}$ must intersect the regular elements non-trivially. Hence there is a regular element whose centraliser is conjugate to $\h_m$. It follows from the existence of the map \eqref{e:kappaiso} that $\ind(\g) = \rank(\g)$, and so $\ind(\g_m, \g_m) = (m+1)\ind(\g)$. Finally the existence of the isomorphism \eqref{e:kappamiso} proves that $\ind(\g_m, \g_m) = \ind(\g_m)$.
\end{proof}

Our next theorem (Theorem~\ref{T:regularelements}) was proven in characteristic zero by Ra\"{i}s and Tauvel, for any finite-dimensional Lie algebra $\g$. Most of \cite[\S2]{RT} still applies in positive charateristic, but a key step fails when $m \geq p$. We provide a proof result in the case where $\g = \Lie(G)$ for a standard reductive group in positive characteristic.

\begin{Lemma}
\label{L:semisimplesdense}
The set of semisimple elements in $\g \backslash \g^\reg$ is dense.
\end{Lemma}
\begin{proof}
This was proven in \cite[Proposition~4.9]{Ve}, under the assumption that $G$ is a semisimple group such that $p$ is good and does not divide the order of the fundamental group of the root system. In view of the classification of Lie algebras of standard reductive groups given in Section~\ref{ss:groupschemes} it remains to deal with the case $\g = \gl_n$ with $p \mid n$.

To prove the claim for $\gl_n$ for any $p > 0$ we use the theory of sheets, which were investigated thoroughly under the standard hypotheses (H1), (H2), (H3) in \cite{PSt}. By Theorem~2.8 of {\it op. op.} there is a dense decomposition class inside every sheet, and by Proposition~2.5 of the same paper, the closure of the sheet contains a unique maximal nilpotent orbit. It follows that $\gl_n \setminus \gl_n^\reg$ is a union of the closures of the sheets which contain the subregular nilpotent orbit. It will suffice to prove that every sheet of $\gl_n$ contains semisimple elements. We formulate Theorem~2.8 of {\it op. cit.} in more detail: for each sheet there is a Levi subalgebra $\g_0 \subseteq \gl_n$ and a rigid nilpotent element $e_0 \in \g_0$ such that $G\cdot (e_0 + \z(\g_0)^{\reg})$ is dense in the sheet. The Levi subalgebras of $\gl_n$ are products of general linear algebras, and so in order to complete the proof it suffices to show that zero is the only rigid nilpotent element in $\gl_n$. Over the complex numbers this is well-known (see \cite[Theorem~7.2.3]{CM} for the case $\sl_n(\C)$), however that argument works equally well for $\gl_n$ in any positive characteristic because the dimensions of centralisers and orbits depend on the partition associated to the orbit, but not on $p$ (see \textsection 2.3 and \textsection 3.1 of \cite{JaNO}, for example).
\end{proof}

\begin{Theorem}
\label{T:regularelements} 
Let $x = \sum_{i=0}^m x_i t^i \in \g_m$. Then $x_0$ is a regular element of $\g$ if and only if $x$ is a regular element of $\g_m$.
\end{Theorem}
\begin{proof}
Let $x = \sum x_i t^i \in \g_m$. Suppose $x_0$ is regular in $\g$. Then since $(\g_m)^{x_0} = (\g^{x_0})_m$, by Proposition \ref{P:genericstabiliser} it is also regular in $\g_m$. Consider the homomorphism $\gamma : \k^\times \to \Aut(\g_m)$ which is defined by $\gamma(s) (xt^i) = s^i xt^i$, known as the group of loop rotations. Note that $x_0 \in \overline{\gamma(\k^\times) x}$. By the upper-semi continuity of stabiliser dimension it follows that $\dim \g_m^{x_0} \ge \dim \g_m^x$. Since $x_0$ is regular in $\g$ we have $\dim \g_m^{x_0} = (m+1)\dim \g^{x_0}$ and so $x$ is regular in $\g_m$ by Lemma~\ref{P:genericstabiliser} and \eqref{e:nondegenform}.

Now suppose that $x_0$ is not regular, and we prove that $x \in \g_m$ is not regular. Thanks to Lemma~\ref{L:semisimplesdense} there is a dense open subset of $\g \setminus \g^\reg$ consisting of semisimple elements. If we can prove that $y = \sum_i y_it^i$ is not regular when $y_0$ lies in this collection of irregular semisimple elements, then it will follow that $x$ is not regular, using the upper-semi continuity of stabiliser dimension again.

Fix $y$ as above and write $y = y_n + y_s$ for the Jordan decomposition. Note that $y_n$ lies in $\g_m^{(\ge 1)}$. Thanks to Proposition~\ref{P:semisimplesincurrents} we can assume, after conjugation by an element of $G_m$, that $y_s$ lies in $\g_m^{(0)}$. Since $G_m$ stabilises $\g_m^{(\ge 0)}$ and $G_m$ preserves the Jordan decomposition, we still have $y_n \in \g_m^{(\ge 1)}$ after conjugation.

By Lemma~\ref{L:Jordandecomposition} we have $\g_m^y = (\g_m^{y_s})^{y_n} = ((\g^{y_s})_m)^{y_n}$, however $\g^{y_s}$ is the Lie algebra of a standard reductive group $G^{y_s}$ of the same rank as $G$, and strictly lower dimension. By an inductive argument, it remains to show that $y_n$ is not a regular element of $\g_m^{y_s}$. This follows from part (2) of Proposition~\ref{P:reductionbydegree} (we note that the proof of that result does not depend on anything which precedes it).
\end{proof}

\subsection{Symmetric invariants and the centre of the enveloping algebra}

In this section we study the invariant theory of the adjoint representation, generalising the work of Kostant (for a good exposition in characteristic zero see \cite[\textsection 3.2]{CG}). The key differences here are that we are working with truncated currents and that the characteristic is positive. Surprisingly, many of the important structural features extend to this setting.

We first recall some properties of the invariant ring $\k[\g_m]^{G_m}$. The structure of this ring is well-understood in characteristic zero, thanks to \cite{RT}. In characteristic zero $\k[\g_m]^{G_m} = \k[\g_m]^{\g_m}$; however in characteristic $p$ this is clearly false because $\k[\g_m]^p \subseteq \k[\g_m]^{\g_m}$.

Thanks to the isomorphism \eqref{e:nondegenform} there is a $G_m$-equivariant isomorphism $\k[\g_m] \cong k[\g_m^*] = S(\g_m)$. There is also a $G_m$-equivariant isomorphism $S(\g) \cong S(\g_m^{(m)})$ where $G_m$ acts on the left-hand side via the surjection $G_m \onto G$. The ring of invariants $S(\g)^G$ is a polynomial ring generated by $r = \dim(\h)$ elements which we denote $p_1,...,p_r$. These give rise to $G_m$-invariants in $S(\g_m)$, which we denote by the same symbols.

In \cite[\textsection 2.2]{CT} we explained that the invariants $p_1,...,p_r$ can be used to build a set of elements $\{p_{i,j} \mid i=1,...,r, \ j = 0,...,m\}$ using divided power derivations. In the notation of {\it loc. cit.} these were denoted $\partial^{(j)}p_i$. They have the property that each $p_{i,j}$ is $G_m$-invariant. These same facts were also explained in \cite[\textsection 4]{ATV} in characteristic $p > 0$, albeit in a rather different language.

\begin{Lemma}
\label{L:symmetricinvariantslemma}
\begin{enumerate}
\item The elements $p_{i, j}$ form an algebraically independent set of generators of the invariant ring $S(\g_m)^{G_m}$.
\item $S(\g_m)^{\g_m}$ is a free module over $S(\g_m)^p$ of rank $p^{(m+1)(\dim \g - \ind \g)}$, with basis given by the restricted monomials
$$\{\prod_{i,j} p_{i,j}^{k_{i,j}} \mid 0 \le k_{i,j} < p\}.$$
In particular $S(\g_m)^{\g_m}$ is generated by $S(\g_m)^p$ and $S(\g_m)^{G_m}$.
\item $S(\g_m)^{\g_m}$ is isomorphic to the tensor product of $S(\g_m)^p$ and $S(\g_m)^{G_m} = \k[p_{i,j} \mid i,j]$ over their intersection.
\end{enumerate}
\end{Lemma}
\begin{proof}
Consider the adjoint $\g[[t]]$-action on $\g((t))$, which extends to an action on $S(\g((t)))$. There is a $\g[[t]]$-stable ideal of $S(\g((t)))$ generated by $\g[[t]]$ and the quotient will be denoted $S(\g[t^{-1}] t^{-1})$. The $\g[[t]]$-module $S(\g[t^{-1}] t^{-1})$ is filtered, such that the $(m + 1)$th filtered piece is equal to 
\begin{eqnarray}
\label{e:filtblah}
S(\oplus_{i=-m-2}^{-1} \g \otimes t^i).
\end{eqnarray}
The natural $G[[t]]$-action and $\g[[t]]$-actions on \eqref{e:filtblah} factor through $G[[t]] \to G_m$ and $\g[[t]] \to \g_m$, respectively, and it is easily seen that \eqref{e:filtblah} is isomorphic to $S(\g_m)$ as a $G_m$-module.

In the proof of \cite[Theorem~4.4]{ATV} all of the claims of the current lemma were proven, in the guise of \eqref{e:filtblah}.
\end{proof}

\begin{Remark}
We should mention that in characteristic zero, the proof that $p_{i,j}$ are $\g_m$-invariant (and hence $G_m$-invariant) is an easy calculation, however in positive characteristic it is a little harder. This fact was proven in \cite[Lemma~4.2]{ATV} under the hypothesis that $p$ is larger than the Coxeter number; however the same proof applies in our setting verbatim (Cf. Remark~1.2 of {\it op. cit.}).
\end{Remark}

We make a brief digression to record a description of the centre $Z(\g_m) = U(\g_m)^{\g_m}$ of the enveloping algebra. We recall from Proposition~\ref{P:Milnersmapandsmoothness}(1) that there is a $G_m$-equivariant (hence $\g_m$-equivariant) isomorphism of filtered vector spaces $U(\g_m) \to S(\g_m)$. Let $z_{i,j}$ denote the preimage of $p_{i,j}$ under this isomorphism.
\begin{Corollary}
\label{C:centreenvalg}
\begin{enumerate}
\item The elements $z_{i, j}$ form an algebraically independent set of generators of the invariant ring $Z(\g_m)^{G_m}$.
\item $Z(\g_m)$ is a free module over $Z_p(\g_m)$ of rank $p^{(m+1)(\dim \g - \ind \g)}$, with basis given by the restricted monomials
$$\{\prod_{i,j} z_{i,j}^{k_{i,j}} \mid 0 \le k_{i,j} < p\}.$$
In particular $Z(\g_m)$ is generated by $Z_p(\g_m)$ and $U(\g_m)^{G_m}$.
\item $Z(\g_m)$ is isomorphic to the tensor product of $Z_p(\g_m)$ and $U(\g_m)^{G_m} = \k[z_{i,j} \mid i,j]$ over their intersection.
\end{enumerate}
\end{Corollary}
\begin{proof}
This follows by a standard filtration argument, which we will not repeat here. We refer the reader to \cite[\textsection 7.2]{Top} for more detail.
\end{proof}
\begin{Remark}
\begin{enumerate}
\item It is easy to see using a simple adaptation of Harish-Chandra's method that there is an injective homomorphism $U(\g_m)^{G_m} \to U(\h_m)$, however describing the image seems to be an interesting challenge.
\item Using the previous theorem it is possible to show that the Azumaya locus on the Zassenhaus variety $\Spec Z(\g_m)$ is equal to the smooth locus, using an argument identical to \cite[Theorem~10]{Top}.
\end{enumerate}

\end{Remark}

\subsection{The transverse slice to the regular orbit}

Let $e \in \g$ be a regular nilpotent element. Thanks to \cite[Lemma~5.3 \&  Proposition~5.8]{JaNO} we can choose a cocharacter $\lambda_e : \k^\times \to G(\k)$ such that the associated grading $\g = \bigoplus_{i \in \Z} \g(i)$ given by $\lambda_e$-weight spaces has the following properties: 
\begin{itemize}
\item[(i)] $e\in \g(2)$;
\item[(ii)] $\g^e \subseteq \g(\ge 0)$.
\end{itemize}
This is analogous to the Dynkin grading attached to nilpotent elements of complex semisimple Lie algebras via $\sl_2$-triples. We define a Kazhdan grading on $\g_m$ by placing $xt^j$ in degree $i - 2$, for $x\in \g(i)$. We transfer this grading to $\k[\g_m]$ via \eqref{e:kappamiso}. Note that the latter grading is just the eigenspace decomposition for the cocharacter
\begin{eqnarray}
\label{e:Kgradingcochar}
\begin{array}{rcl}
\gamma_e &:& \k^\times \to \GL(\g_m);\\
& & t \mapsto t^2 \lambda_e(t).
\end{array}
\end{eqnarray}

Choose a graded complement to $[\g,e]$ in $\g$, and denote it $\v$. The affine subspace $\Ss_e := e + \v$ is known as the {\it good transverse slice to $e$}: it is the positive characteristic analogue of the Slodowy slice (see \cite[\textsection 5.1]{GTmod}). Since $\gamma_e$ preserves $\Ss_e$ and the weights are all negative it follows that the induced grading on $\k[\Ss_e]$ is non-negative and connected. It is also clear that the $\k^\times$-action is contracting with unique fixed point $e$.

Now consider $\Ss_{e,m} := e + \v_m$, which is the analogue of the Slodowy slice in the truncated current algebra $\g_m$ (this is similar to \cite[\textsection 4]{RT}). The natural homomorphism $\k[\g_m] \onto \k[e + \v_m]$ has Kazhdan graded kernel, and $\k[e + \v_m]$ inherits a non-negative, connected grading, which we call the {\it Kazhdan grading}. Again, the contracting action on $\Ss_{e,m}$ given by $\gamma_e$ has unique fixed point $e$. Since $\gamma_e$ is a composition of an inner automorphism with the cocharacter $t\mapsto t^2$, it follows that if $f \in S(\g_m)^{G_m}$ has total degree $d$ then it has Kazhdan degree $2d$.

We also wish to consider the invariant ring $\k[\g_m]^{G_m}$, which thanks to \eqref{e:kappamiso} can be described using Lemma~\ref{L:symmetricinvariantslemma}. Consider the adjoint quotient map
\begin{eqnarray}
\label{e:adjquot}
\pi_m : \g_m \longrightarrow \g_m /\!/ G_m := \Spec \k[\g_m]^{G_m}.
\end{eqnarray}

By viewing $\k[\g_m]$ as polynomials on a vector space we equip $\k[\g_m]$ with a non-negative grading by total degree (i.e. $\g_m^* \subseteq \k[\g_m]$ lies in degree 1). Clearly $\k[\g_m]^{G_m}$ is a graded subalgebra, and this induces a contracting $\k^\times$-action on both $\g_m$ and $\g_m /\!/ G_m$ such that $\pi_m$ is $\k^\times$-equivariant. We will denote the $\k^\times$-fixed point in $\g_m /\!/ G_m$ by $0$.

The following result has its roots in a famous theorem of Kostant. In characteristic zero the truncated current version is \cite[Theorem~4.5(1)]{RT}, whilst the in positive characteristic the $m = 0$ case is a result of Simon Riche \cite[Theorem~3.2.2]{Ri}.

\begin{Theorem}
\label{T:sliceisamazing}
The following hold:
\begin{enumerate}
\item The slice $e + \v_m$ intersects each regular orbit in a single point.
\item The restriction map $\k[\g_m]^{G_m} \to \k[e + \v_m]$ is an isomorphism.
\end{enumerate}

\end{Theorem}
\begin{proof}
The $\k^\times$-action on $\g_m$ normalises the $G_m$-action and it follows that the dimension of an orbit is invariant under $\k^\times$-dilations. Since $\k^\times$ preserves $e + \v_m$ and $e \in \g_m$ is regular, it follows that $e + \v_m$ consists of regular elements by upper semicontinuity.

Now suppose that $x = \sum x_it^i \in \g_m$ is a regular element. By Theorem~\ref{T:regularelements} we know that $x_0 \in \g^\reg$, whilst by \cite[Theorem~3.2.2]{Ri} we know that $G\cdot (e + \v)  = \g^\reg$. Therefore, after conjugating by an element of $G$, which acts diagonally on $\g_m$, we can assume that $x_0 \in e + \v$. Now applying the same line of reasoning as in the proof of \cite[Lemma~4.2]{RT}, using Lemma~\ref{L:groupstructure} in place of the exponential mappings appearing there, we deduce that there is an element of $G_m$ conjugating $x$ to $e + \v_m$. This proves (1).

By (1) it follows that the restriction map in part (2) is injective.  To complete the proof we show that the Hilbert series (recording Kazhdan graded dimensions) of both algebras are the same.

Let $p_1, \dots, p_r$ be homogenous generators of $\k[\g]^G$ of total degrees $n_1, \dots, n_r$ respectively (and hence Kazhdan degrees $2n_1, \dots 2n_r$). Then by Lemma \ref{L:symmetricinvariantslemma} we obtain a set of homogenous generators $p_{i, j}$ for $\k[\g_m]^{G_m}$ such that $p_{i, j}$ lies in Kazhdan degree $2n_i$. On the other hand, by the $m = 0$ case \cite[Theorem~3.2.2]{Ri}, we have generators $q_1, \dots, q_r$ of $\k[e + \v]$ lying in Kazhdan degrees $2n_1, \dots, 2n_r$. Identifying $\k[e + \v_m] = J_m \k[e+\v]$ we have generators generators $q_{i, j} = \partial^{(j)} q_i \in \k[e + \v_m]$ with $q_{i, j}$ lying in Kazhdan degree $2n_i$, completing the proof.
\end{proof}

\section{Representation theory of truncated currents}
\label{S:repthy}

\subsection{Reduction by truncation degree}

Fix $0 \le k \le m$. We say that {\it $\chi \in \g^*$ is supported in degree less than or equal to $k$} if 
\begin{eqnarray}
\chi\Big(\g_m^{(>k)}\Big) = 0.
\end{eqnarray}
Using \eqref{e:kappamiso} we see that every $\chi\in \g^*$ can be expressed uniquely as $\bkappa_m(x)$ for some $x\in \g$. Then $\chi$ is supported in degree less than or equal to $k$ if and only if $x \in \g_m^{(\ge m-k)}$. 

The following result will be useful in classifying simple $\g_m$-modules.
\begin{Proposition}
\label{P:reductionbydegree}
Let $\chi \in \g_m^*$ be supported in degree less than or equal to $k$, and let $\psi = \chi|_{\g_k} \in \g_k^*$, where we identify $\g_k$ with $\g_m^{(\leq k)} \subseteq \g_m$ as vector spaces.
\begin{enumerate}
\item The full subcategory of $U_\chi(\g_m)\lmod$ whose objects are the modules annihilated by $\g_m^{(> k)}$ is equivalent to $U_\psi(\g_k)\lmod$. This full subcategory contains all simple $U_\chi(\g_m)$ modules.

\item $\dim \g_m - \dim \g_m^\chi = \dim \g_k - \dim \g_k^\psi.$
\end{enumerate}
\end{Proposition}

\begin{proof}
Let $I_k$ be the ideal of $U_\chi(\g_m)$ generated by $\g_m^{(> k)}$. Using the PBW theorem for reduced enveloping algebras it is not hard to see that the map $\g_m \to U_\chi(\g_m)/I_k$ factors through $\g_m \onto \g_k$ and the map $\g_k \to U_\chi(\g_m) / I_k$ induces an isomorphism $U_\psi(\g_k) \cong U_\chi(\g_m)/I_k$. Furthermore $I_k$ is generated by nilpotent elements and so, since $U_\chi(\g_m)$ is artinian, $I_k$ is contained in the Jacobson radical \cite[Theorem~0.1.12]{MR}. It follows that every simple $U_\chi(\g_m)$-module is annihilated by $I_k$. This completes the proof of (1).

Now observe that $\g_m^{(>k)}$ is an ideal of $\g_m$ on which $\chi$ is 0, so $\g_m^{(>k)} \subseteq \g_m^\chi$ and hence $\g_m^\chi + \g_m^{(\leq k)} = \g_m$. Now $\dim(\g_m) = \dim(\g_m^\chi + \g_m^{(\leq k)}) = \dim(\g_m^\chi) + \dim(\g_m^{(\leq k)}) - \dim(\g_m^\chi \cap \g_m^{(\leq k)})$, so it is enough to show that $\g_m^\chi \cap \g_m^{(\leq k)} = \g_k^\psi$ (again identifying $\g_k$ with $\g_m^{(\leq k)}$). But this follows from the fact $\g_m / \g_m^{(>k)} \cong \g_k$ and the definition of $\psi$, completing the proof.
\end{proof}

\subsection{A Morita equivalence by parabolic induction}

We now prove a parabolic induction theorem for $\g_m$, which allows us to reduce the problem of classifying $U_\chi(\g_m)$-modules to the case where $\chi$ is nilpotent. It is very similar in spirit and proof to a famous result of Friedlander and Parshall \cite[Theorem~3.2]{FPmod}. We start with the following result, which is easy to prove but vital to our later reductions.
\begin{Lemma}
\label{L:reducingoverthecentre}
Suppose $\chi \in \g_m^*$ satisfies $\chi([\g_m, \g_m]) = 0$. Then $U_\chi(\g_m)\lmod \cong U_0(\g_m)\lmod$.
\end{Lemma}
\begin{proof}
Let $\lambda \in \g_m^*$ be the unique linear function satisfying $\chi(x)^p - \chi(x^{[p]}) = \lambda(x)^p$ for all $x\in \g_m$. This $\lambda$ certainly exists; for all $x \in \g_m$, set $\lambda(x)$ to be the unique $p$-th root of $\chi(x)^p - \chi(x^{[p]})$. Since $[\g_m, \g_m]$ is a restricted Lie subalgebra, it follows that $\lambda([\g_m, \g_m]) = 0$ and so $\lambda$ defines a one dimensional $\g_m$-module $\k_\lambda$ with $p$-character $\chi$. It is then not hard to see that $M \mapsto M \otimes_{U(\g_m)} \k_\lambda$ is an equivalence $U_0(\g_m)\lmod \to U_\chi(\g_m)\lmod$.
\end{proof}

Suppose $\chi$ is a $p$-character with Jordan decomposition $\chi = \chi_s + \chi_n$. By Proposition~\ref{P:semisimplesincurrents}, after conjugating by an element of $G_m$ we may suppose that $\chi_s$ vanishes on $\h_m^{(<m)} \oplus \n_m^\pm$. The centraliser $\g^{\chi_s}$ is a Levi subalgebra of $\g$ and, since $\chi_s$ is supported on $\g_m^{(m)}$, we have $(\g_m)^{\chi_s} = (\g^{\chi_s})_m$ and hence may write $\g_m^{\chi_s}$ unambiguously.

Let $\p = \g^{\chi_s} \oplus \r$ be a parabolic subalgebra of $\g$ with nilradical $\r$. Note that $\chi(\r_m) = 0$ by Lemma~\ref{L:Jordandecomposition}, and so $U_\chi(\g_m^{\chi_s})$-modules can be naturally inflated to $U_\chi(\p_m)$-modules by letting $\r_m$ act trivially.

 Consider the functors
\begin{eqnarray}
\label{e:inductionfunctor}
\begin{array}{rcl}
\Ind & : & U_\chi(\g^{\chi_s}_m)\lmod \longrightarrow U_\chi(\g_m)\lmod,\\
& & M \longmapsto U_\chi(\g_m) \otimes_{U_{\chi}(\p_m)} M.
\end{array}\\
\label{e:invariantsfunctor}
\begin{array}{rcl}
(\bullet)^{\r_m} &:& U_\chi(\g_m)\lmod \longrightarrow U_\chi(\g_m^{\chi_s})\lmod,\\
& & M \longmapsto M^{\r_m}.
\end{array}
\end{eqnarray}

\begin{Theorem}
\label{T:parabolicinduction}
The functors \eqref{e:inductionfunctor} and \eqref{e:invariantsfunctor} are quasi-inverse equivalences of categories.
\end{Theorem}

\begin{proof}
Our approach is similar to \cite[7.4]{JaLA} (which covers the $m = 0$ case) however we provide the details for the reader's convenience.

First observe that each $U_\chi(\g_m)$-module $M$ is free as a $U_0(\r_m)$-module. Indeed, by Theorem \ref{T:supportvarieties} we have $\V_{\r_m}(M) \subseteq \N^{[p]}(\g_m^\chi) \cap \r_m = 0$, and hence $M$ is projective as a $U_\chi(\r_m)$-module. Furthermore since $\chi([\r_m, \r_m]) = \chi(\r_m^{[p]}) = 0$, the argument from \cite[Corollary 7.2]{JaLA} shows that that in fact $M$ is free as a $U_\chi(\r_m)$-module.

It now follows that both functors \eqref{e:inductionfunctor} and \eqref{e:invariantsfunctor} are exact. The algebra $U_0(\r_m)$ is a Frobenius algebra \cite[Proposition~1.2]{FPmod} and hence has a simple socle, so $\dim U_0(\r_m)^{\r_m} = 1$. Since $\dim(\r) = \dim(\g/\p)$, we have that $\dim M = p^{\dim \r_m}\dim M^{\r_m}$ and $\dim (U_\chi(\g_m) \otimes_{U_\chi(\p_m)} N) = p^{\dim \r_m} \dim N$. In particular $\dim M  = \dim (U_\chi(\g_m) \otimes_{U_\chi(\p_m)} M^{\r_m})$ and $\dim N = \dim (U_\chi(\g_m) \otimes_{U_\chi(\p_m)} N)^{\r_m}$.

To complete the proof we must show that the units of the natural adjunction morphisms are isomorphisms. In particular, we wish to show that the maps $f_M: U_\chi(\g_m) \otimes_{U_\chi(\p_m)} M^{\r_m} \rightarrow M$ and $g_N: N \rightarrow (U_\chi(\g_m) \otimes_{U_\chi(\p_m)} N)^{\r_m}$ given by $f(u \otimes m) = u \cdot m$ and $g(n) = 1 \otimes n$ are isomorphisms. We first observe that $g_N$ is certainly injective, and so by considering dimensions is an isomorphism. Now we show that $f_M$ is an isomorphism for $M$ simple: it must be surjective as it has non-zero image, and so considering dimensions once again it is an isomorphism. Since our functors are exact and all categories in sight are artinian, a standard argument using the short Five Lemma then shows $f_M$ is an isomorphism for any $M$.
\end{proof}

\begin{Corollary}
\label{C:semisimplecharactermodules}
If $\chi$ is a regular semisimple $p$-character then
$$U_\chi(\g_m) \cong \Mat_{p^{(m+1)\ind(\g)}} U_0(\h_m)$$
\end{Corollary}
\begin{proof}
Using \cite[Proposition~2.5]{PrST} and Theorem~\ref{T:parabolicinduction} we have $U_\chi(\g_m) \cong \Mat_{p^{(m+1)\ind(\g)}} U_\chi(\h_m)$. The equivalence described in Lemma~\ref{L:reducingoverthecentre} induces an algebra isomorphism $U_0(\h_m) \cong U_\chi(\h_m)$.
\end{proof}

We now describe the simple and projective modules over $U_0(\h_m)$. Since $\h_m$ is commutative we have that the simple modules are the 1-dimensional modules $\k_\lambda$ for $\lambda \in \h^*$ such that $\lambda(h)^p = \lambda(h^{[p]})$, where $\h_m^{(0)}$ acts by $\lambda$ and $\h_m^{(\geq 1)}$ acts by 0. Note that there are $p^{\dim(\h)}$ such $\lambda$. Each simple module $\k_\lambda$ then has a unique projective cover $Q^{\h_m}(\lambda)$, and these $Q^{\h_m}(\lambda)$ form a complete set of irreducible projective modules for $U_0(\h_m)$.

\begin{Proposition}
\label{P:torusprojectives}
The only composition factor of $Q^{\h_m}(\lambda)$ is $\k_\lambda$, which occurs with multiplicity $p^{m \dim(\h)}$.
\end{Proposition}

\begin{proof}
Consider the module $M(\lambda):= U_0(\h_m) \otimes_{U_0(\h_m^{(0)})} \k_\lambda$. As a $U_0(\h_m^{(\geq 1)})$-module this is isomorphic to $U_0(\h_m^{(\geq 1)})$ so, since $\h_m^{(\geq 1)}$ is $p$-nilpotent, $M(\lambda)$ is indecomposable by \cite[Corollary 3.4]{JaLA}. Hence $\dim(Q^{\h_m}(\lambda)) \geq \dim(M(\lambda)) = p^{m\dim(\h)}$. But $p^{(m+1)\dim(\h)} = \dim(U_0(\h_m)) \geq \sum_\lambda \dim(Q^{\h_m}(\lambda)) \geq \sum_\lambda \dim(M(\lambda)) = p^{\dim(\h)} p^{m \dim(\h)}$, so in fact $Q^{\h_m}(\lambda) \cong M(\lambda)$. Hence we see that $\h_m^{(0)}$ acts by $\lambda$ on $Q^{\h_m}(\lambda)$, so the only composition factor of $Q^{\h_m}(\lambda)$ is $\k_\lambda$ and since $\dim(Q^{\h_m}(\lambda)) = \dim(M(\lambda)) = p^{m \dim(\h)}$ it must occur with multiplicity $p^{m \dim(\h)}$ as required.
\end{proof}

\section{The Kac--Weisfeiler conjectures}
\label{S:KWconj}

In \cite{KW}, Kac--Weisfeiler made the following conjectures:

\begin{Conjecture}[KW1]
Let $\g$ be a restricted Lie algebra and let $M(\g)$ be the maximal dimension of a simple $U(\g)$-module. Then $M(\g) = p^{\frac{1}{2}(\dim(\g) - \ind(\g))}$.
\end{Conjecture}

\begin{Theorem*}[\cite{PrIR}, KW2]
Let $\g$ be the Lie algebra of a reductive group $G$ under the standard hypotheses. Then for any $p$-character $\chi$ and any simple $U_\chi(\g)$-module $M$, $p^{\frac{1}{2}(\dim \g - \dim \g^\chi)}$ divides $\dim M$.
\end{Theorem*}

The conjecture KW1 is still open in general, but some results for certain classes of Lie algebras are known. In their original paper \cite{KW}, Kac and Weisfeiler proved that KW1 holds for restricted completely solvable Lie algebras, and in \cite{PSk} Premet and Skryabin showed that it holds for any restricted $\g$ such that there exists $\chi \in \g^*$ with $\g^\chi$ a torus. Neither of these results apply to truncated current Lie algebras; they are certainly not completely solvable, and any centraliser in $\g_m$ must intersect the nilpotent subalgebra $\g_m^{(m)}$ non-trivially.

A result of Martin, Stewart and the second author \cite[Theorem 1.1]{MST} states that for fixed $n$, there exists $p_0$ such that KW1 holds for any restricted subalgebra of $\gl_n(\k)$ where $\k$ is an algebraically closed field of characteristic greater than $p_0$. This implies that if we fix $m$ and the group scheme $G$, then in sufficiently large characteristic KW1 holds for $\g_m$. However, this result does not give an explicit bound on $p_0$. Another result of the second author \cite[Theorem 4]{Top} states that KW1 holds for any centraliser of a nilpotent element in $\gl_k$. This implies KW1 holds for $(\gl_n)_{m}$ since, as observed in \cite{Ya}, if $e \in \gl_{(m+1)n}$ is a nilpotent element corresponding to the rectangular partition $(m+1, m+1, \dots, m+1)$ then $(\gl_{(m+1)n})^e \cong (\gl_n)_m$. We give a proof of (KW1) for $\g_m$ which is valid when $\g$ is the Lie algebra of a reductive group $G$ requiring no assumptions beyond the standard hypotheses.

In \cite{PrIR} Premet proved KW2, subsequently giving other proofs in \cite[\textsection 2.6]{PrST} and \cite[Theorem~5.6]{PSk}. In \cite{Kac}, Kac conjectured that the statement KW2 holds for all Lie algebras of algebraic groups. Very little is known about this statement outside the standard reductive case. The only exception is the case of completely solvable Lie algebras, where it can be extracted from the classification of simple modules given in \cite{KW}. We show that KW2 holds for certain classes of $p$-characters of the truncated current Lie algebra $\g_m$, and also prove KW2 for all $p$-characters of $(\sl_2)_m$ in characteristic greater than 2.

\subsection{The first Kac-Weisfeiler conjecture}

Recall that earlier in the proof of Proposition \ref{P:genericstabiliser} we defined $\h_m^\reg := \{ \sum x_i t^i \in \h_m \mid x_0 \in \h^\reg\}$.

\begin{Proposition}
\label{P:hregmodules}
Let $\chi = \bkappa_m(h)$ for some $h \in \h_m^\reg$. Then the algebra $U_\chi(\g_m)$ has precisely $p^{\rank(\g)}$ simple modules and projective indecomposable modules, up to isomorphism. The simple modules have dimension $p^{\frac{m+1}{2}(\dim(\g) - \rank(\g))}$ while the projective modules have dimension $p^{\frac{m+1}{2}(\dim(\g) + \rank(\g)) - \rank(\g)}$.
\end{Proposition}

\begin{proof}
Let $h = \sum_{x=0}^m x_i t^i$. Observe that the Jordan decomposition of $h$ has semisimple part $x_0$ and nilpotent part $\sum_{x=1}^m x_i t^i$, so by Theorem \ref{T:parabolicinduction} we obtain an equivalence between $U_\chi(\g_m)$ and $U_\chi(\h_m)$. Furthermore, by Lemma \ref{L:reducingoverthecentre} together with Proposition \ref{P:torusprojectives} and the preceeding discussion, the simple modules for $U_\chi(\h_m)$ are 1-dimensional, the projective modules are $p^{m\rank(\g)}$-dimensional, and there are $p^{\rank(\g)}$ of each. Now, from the proof of Theorem \ref{T:parabolicinduction}, for any $U_\chi(\h_m)$-module $N$ we have $\dim (U_\chi(\g_m) \otimes_{U_\chi(\b_m)} N) = p^{\dim \b_m} \dim N = p^{\frac{m+1}{2}(\dim(\g) - \rank(\g))} \dim N$ and the desired result follows.
\end{proof}

\begin{Corollary}
\label{C:KW1forgm}
The first Kac-Weisfeiler conjecture holds for $\g_m$ where $\g$ is the Lie algebra of a standard reductive algebraic group.
\end{Corollary}

\begin{proof}
By the proof of Proposition \ref{P:genericstabiliser}, the conjugates of $\h_m^\reg$ are dense in $\g_m$. Recall that we defined $M(\g_m)$ to be the maximal dimension of a simple $U(\g_m)$-module. By \cite[Proposition 4.2(1)]{PSk} the set of $p$-characters $\chi$ such that all simple $U_\chi(\g_m)$-modules have dimension $M(\g_m)$ is non-empty and open, and so in particular intersects the conjugates of $\bkappa_m(\h_m^\reg)$ non-trivially. Now by Proposition \ref{P:hregmodules}, if $\chi$ is conjugate to an element of $\bkappa_m(\h_m^\reg)$, then all simple $U_\chi(\g_m)$-modules have dimension $p^{\frac{m+1}{2}(\dim(\g) - \rank(\g))}$ which is equal to $p^{\frac{1}{2}(\dim(\g_m) - \ind(\g_m))}$ by Theorem \ref{T:regularelements}. Hence $M(\g_m) = p^{\frac{1}{2}(\dim(\g_m) - \ind(\g_m))}$.
\end{proof}

\subsection{The second Kac-Weisfeiler conjecture}

We now investigate the second Kac-Weisfeiler conjecture for the case of truncated currents on the Lie algebra of a general standard reductive group. In particular, although we do not show the statement holds for all $p$-characters, we give several families for which it does hold. Following \cite[\S2.3]{PrST}, for $\chi \in (\g_m)^*$ a nilpotent $p$-character, we define a $\chi$-admissible subalgebra of $\g_m$ to be a subalgebra $\m$ satisfying:
\begin{enumerate}
\item $\m$ consists of nilpotent elements.
\item $\chi(\overline{[\m, \m]}) = 0$, where the over-line denotes the $[p]$-closure. 
\item $\V_{\g_m}(\chi) \cap \m = 0$.
\end{enumerate}
By \cite[Theorem 2.3(ii)]{PrST}, if there exists a $\chi$-admissible subalgebra $\m$ such that $\dim(\m) = \frac{1}{2}(\dim \g_m - \dim \g_m^\chi)$ then KW2 holds for $\chi$.

\begin{Proposition}
\label{P:KW2pcharacters}
Let $\g$ be the Lie algebra of a standard reductive group. Then for $\g_m$, we have:
\begin{itemize}
\item[(i)] KW2 holds for all homogeneous $p$-characters (this includes for example the minimal orbit).
\item[(ii)] KW2 holds for all semisimple $p$-characters.
\item[(iii)] KW2 holds for all regular $p$-characters.
\end{itemize}
\end{Proposition}

\begin{proof}
Parts (i) and (ii) follow easily from Proposition \ref{P:reductionbydegree} and Theorem \ref{T:parabolicinduction} respectively. 
For part (iii), let $\chi = \bkappa_m(x)$ for some regular $x = \sum_{i = 0}^m x_i t^i$. Now, by Theorem \ref{T:parabolicinduction} we may reduce to the case where $x$ is regular nilpotent, and then by Proposition \ref{P:nilpotentsincurrents} and Theorem \ref{T:regularelements} we see that $x_0$ is regular nilpotent. Let $\chi_0 = \bkappa(x_0) \in \g^*$. Observe that if $\b \subseteq \g$ is a Borel containing $x_0$, then the nilradical $\n^-$ of the opposite Borel $\b^-$ is a $\chi_0$-admissible subalgebra of dimension $\frac{1}{2}(\dim(\g) - \dim(\g^{\chi_0})) = \frac{1}{2}\dim(G \cdot \chi_0)$. Also, we have $[\n, \n] \subseteq \ad(\g) x_0$ so by Lemma \ref{L:groupstructure}, after conjugating $x$ by a suitable element of $G_m$ we may assume that $\kappa(x_i, n) = 0$ for all $n \in [\n^-, \n^-] = \overline{[\n^-, \n^-]}$.

We now claim that $\n_m^-$ is a $\chi$-admissible subalgebra of $\g_m$, which will complete the proof since $\dim(\n_m^-) = \frac{(m+1)}{2}(\dim(\g) - \ind(\g)) = \frac{1}{2}(\dim \g_m - \dim \g_m^\chi)$. Condition (1) follows from the fact $\n$ is a $\chi_0$-admissible subalgebra. To see that (2) holds, observe that $\chi(y) = \kappa_m(x_0, y) + \kappa_m(\sum_{i = 1}^m x_i t^i, y)$ for any $y = \sum_{i = 0}^m y_i t^i\in \overline{[\n_m^-, \n_m^-]}$, and the first term is zero since $\chi_0(\overline{[\n^-, \n^-]}) = 0$ while the second is zero since by assumption $\kappa_m(x_i, y_i) = 0$. For condition (3), observe that if $n = \sum_{j = i}^m n_i t^i \in \n_m^-$ then $[x, n] = [x_0, n_i] t^i + \sum_{j = i + 1}^m z_j t^j$ for some $z_j \in \g$, which is non-zero since $\n \cap \g^{\chi_0} = 0$. Hence $\n_m \cap \g_m^\chi = 0$ so by Corollary \ref{C:supportvarieties} condition (3) is satisfied.
\end{proof}

\begin{Proposition}
\label{P:KW2forsl2}
The second Kac-Weisfeiler conjecture holds for $(\sl_2)_m$ provided $p > 2$.
\end{Proposition}

\begin{proof}
Let $x = \sum_{i = 0}^m x_i t^i \in (\sl_2)_m$ and let $\chi = \bkappa_m(x)$. Every non-zero element of $\sl_2$ is regular so either $x_0$ is regular and hence $x$ is regular by Theorem \ref{T:regularelements}, or $x_0 = 0$. In the former case, KW2 holds for $\chi$ by Theorem \ref{P:KW2pcharacters}(iii) and in the latter case it holds by Proposition \ref{P:reductionbydegree} and induction on $m$. 
\end{proof}

\section{Classifying simple modules with homogeneous nilpotent $p$-characters}
\label{S:classifyingsimples}
\subsection{Baby Verma modules and their simple quotients}
\label{ss:simpleclassification}

The aim of this section is to classify the simple $U_\chi(\g)$-modules for $p$-characters $\chi$ which are homogeneous and nilpotent of standard Levi type.

In order to justify our restrictions on $p$-characters we would like to make some basic observations about the role of baby Verma modules. In the case $m = 0$ they are one of main tools for explicit construction of simple modules. Our approach here is to pick a Borel subalgebra $\b^+\subseteq \g$, inflate one dimensional $U_\chi(\h_m)$-modules to $U_\chi(\b^+_m)$ and then induce. Note that this is only possible if we can choose a $G_m$-conjugate of $\chi$ which vanishes on $\n_m^+$ where $\n^+ \subseteq \b^+$ is the radical. When $m>0$ this is not always the case: consider, for example, the case where $\g = \sl_2$ and $m = 1$, and let $\{e, h, f\} \subseteq \sl_2$ the standard basis. Then for $\chi = \bkappa_m(e + ft)$ there is no choice of $\b^+$ such that $\chi(\n^+_m) = 0$.

Now assume there exists (and fix) a choice of $\b^+$ with $\chi(\n^+_m) = 0$ then we can define the {\it baby Verma modules}. Let $\b^+ = \h \oplus \n^+$ where $\h$ a maximal torus of $\g$. 
Define
\[\Lambda_\chi = \{\lambda \in \h_m^*: \lambda(h t^i)^p - \lambda((h t^i)^{[p]}) - \chi(h t^i)^p = 0: h \in \h, 0 \leq i \leq m\} \subseteq \h_m^*\]
Then for any $\lambda \in \Lambda_\chi$, the baby Verma module $Z_\chi(\lambda)$ for $U_\chi(\g_m)$ is given by
\[Z_\chi(\lambda) :=  U_\chi(\g_m) \otimes_{U_\chi(\b^+_m)} \k_\lambda\]
where $\k_\lambda$ is the 1-dimensional module on which $\h_m$ acts by $\lambda$ and $\n_m^+$ acts by 0.

\begin{Lemma}
\label{L:babyvermas1}
Let $\b^+$ be a Borel subalgebra of $\g$ with nilradical $\n^+$, and let $\chi \in \g^*$ be such that $\chi(\n^+) = 0$. Then every simple $U_\chi(\g_m)$-module is the homomorphic image of $Z_\chi(\lambda)$ for some $\lambda \in \Lambda_\chi$.
\end{Lemma}

\begin{proof}
The proof of the analogous statement \cite[6.7]{JaLA} is still valid here, replacing $\n^+, \b^+$ and $\n^-$ with $\n^+_m, \b^+_m$ and $\n^-_m$ respectively.
\end{proof}

We now look at the case where $m > 0$ and $\chi$ is a homogeneous $p$-character of degree $k$, which means that $\chi(\bigoplus_{i\ne k} \g(i)) = 0$. In this case, we can always find a Borel $\b^+$ such that $\chi(\b^+_m) = 0$. In fact by Proposition \ref{P:reductionbydegree}, without loss of generality we can assume $k = m$ here, so let $\chi = \bkappa_m(x)$ for some $x \in \g \cong \g_m^{(0)}$. Then by Theorem \ref{T:parabolicinduction} and Proposition \ref{P:nilpotentsincurrents} we can also assume that $x$ is a nilpotent element of $\g$.

We focus in particular on the case where $\chi = \bkappa_m(e)$ for some $e \in \g$ of {\it standard Levi type}, which means that $e$ is a regular nilpotent element of some Levi subalgebra of $\g$, see \cite[\textsection 10]{JaLA} for more detail.

For the rest of this section, we fix such an $e \in \g$ and $\chi =  \bkappa_m(e)$. We can choose a maximal torus $\h = \Lie(T) \subseteq \g$, let $\Phi$ be the root system with respect to $T$. We choose of simple roots $\Delta$ such that the minimal Levi subalgebra containing $e$ has root system generated by a subset of $\Delta$.

Pick root vectors $e_\alpha$ for $\alpha \in \Phi$ such that $e = \sum_{\alpha \in I} e_\alpha$ for some $I \subseteq \Delta$. Let $\b^+ \subseteq \g$ be the Borel subalgebra corresponding to this choice of simple roots, $\b^-$ the opposite Borel corresponding to $-\Delta$, and $\n^-, \n^+$ their respective nilradicals. 

Finally, we let $\g_I$ be the Levi subalgebra $\h \oplus \bigoplus_{\alpha \in \Z I} \g_{\alpha}$ of $\g$ corresponding to the subset $I \subseteq \Delta$, and let $\r$ and $\r^-$ be the nilradicals of the parabolic subalgebras $\g_I + \b^+$ and $\g_I + \b^-$ respectively.

\begin{Proposition}
\label{P:babyvermaparameterisation}
Let $\chi = \bkappa_m(e)$ for some $e \in \g$ in standard Levi type, and fix a toral basis $\{h_1, \dots, h_r\}$ of $\h$. Then $\Lambda_\chi = \{\lambda \in \h^* : \lambda(\h^{(\geq 1)}) = 0, \lambda(h_j) \in \Fbb_p \subset \k\}$.
\end{Proposition}

\begin{proof}
Consider first the equation $\lambda(h t^i)^p - \lambda((h t^i)^{[p]}) - \chi(h t^i)^p = 0$ in the case $i = m$. Here there is a unique solution $\lambda(h t^m) = \chi(h t^m) = 0$ since $(h t^i)^{[p]} = 0$. We then see inductively for $i = m-1, m-2, \dots, 1$ that $\lambda((h t^i)^{[p]}) = 0$ and so again $\lambda(h t^i) = \chi(h t^i) = 0$. Now, if $h \in \h$ is such that $h^{[p]} = h$, then since $\chi(h) = 0$ we have that $\lambda(h)^p - \lambda(h^{[p]}) = \lambda(h)^p - \lambda(h) = 0$ and so $\lambda(h_j)$ can be chosen to be any value in $\Fbb_p$ for all $1 \leq j \leq r$.
\end{proof}

In light of this, from here on we regard elements of $\Lambda_\chi$ as elements of $\h^*$ rather than of $\h^*_m$. For $\chi$ in standard Levi form, we have the following result analogous to \cite[10.2]{JaLA} which gives a relationship between the baby Verma modules and the simple $U_\chi(\g_m)$-modules.

\begin{Lemma}
\label{L:babyvermas2}
Each $Z_\chi(\lambda)$ has a unique maximal submodule and hence a unique simple quotient $L_\chi(\lambda)$.
\end{Lemma}

\begin{proof}
Just as for Lemma~\ref{L:babyvermas1}, the proof of the analogous statement \cite[10.2]{JaLA} is still valid here, replacing $\n^+, \b^+$ and $\n^-$ with $\n^+_m, \b^+_m$ and $\n^-_m$ respectively.
\end{proof}

By Theorem~\ref{T:parabolicinduction}, along with Lemmas \ref{L:babyvermas1} and \ref{L:babyvermas2},  it suffices to determine when $L_\chi(\lambda) \cong L_\chi(\mu)$ for $\lambda, \mu \in \Lambda_\chi$, in order to classify the simple $U_\chi(\g_m)$-modules for $\chi$ in standard Levi form, supported in degree $m$. We first consider the case $I  = \Delta$, i.e. when $e$ is a regular nilpotent element of $\g$ and $\g_I = \g$.

\begin{Proposition}
\label{P:simplevermas}
Let $\chi = \bkappa_m(e)$ for $e \in \g$ regular nilpotent. Then $Z_\chi(\lambda)$ is simple for all $\lambda \in \Lambda_\chi$.
\end{Proposition}

\begin{proof}
Letting $\m = \n^-_m$ and observing that $\g_m^\chi \cap \n^-_m = 0$, we can apply Theorem~\ref{T:supportvarieties} to see that $\m$ is a $\chi$-admissible subalgebra (note that $\chi$ vanishes on $[\m,\m]$ since we have assumed that $e$ is in standard Levi form). It follows by the same argument as we used in the second paragraph of the proof of Theorem~\ref{T:parabolicinduction} that any $U_\chi(\g_m)$-module is free as a $U_\chi(\m_m)$-module.

Let $N$ be a $U_\chi(\g_m)$ submodule of $Z_\chi(\lambda)$. By dimensional considerations $N$ is free of rank 0 or 1 over $U_\chi(\m_m)$. Hence $N = 0$ or $N = Z_\chi(\lambda)$, i.e. $Z_\chi(\lambda)$ is simple.
\end{proof}

\begin{Proposition}
\label{P:regnilpotentsimples}
Let $\chi = \bkappa_m(e)$ for $e \in \g$ regular nilpotent. Then $L_\chi(\lambda) \cong L_\chi(\mu)$ if and only if $\lambda|_{\z(\g)} = \mu|_{\z(\g)}$.
\end{Proposition}

\begin{proof}
Fix some $\lambda \in \Lambda_\chi$ and $\alpha \in \Delta$, and consider the element $f_\alpha t^m \otimes 1_\lambda \in Z_\chi(\lambda)$ for some $f_\alpha \in \g_{ -\alpha}$. This element is highest weight since if $x \in \n^+_m$ then $[x, f_\alpha t^m] \in (\b^+_m)^{(m)}$, so $x \cdot (f_\alpha t^m \otimes 1_\lambda) = f_\alpha t^m \otimes (x \cdot 1_\lambda) + 1 \otimes ([x, f_\alpha t^m] \cdot 1_\lambda) = 0$. Furthermore, it generates $Z_\chi(\lambda)$ since $(f_\alpha t^m)^{p-1} \cdot (f_\alpha t^m \otimes 1_\lambda) = \chi(f_\alpha t^m)^p \otimes 1_\lambda$ and $\chi(f_\alpha t^m) \neq 0$. Hence we have a surjective homomorphism from $Z_\chi(\lambda - d\alpha)$ to $Z_\chi(\lambda)$, which by considering dimensions must in fact be an isomorphism. But we also have that $Z_\chi(\lambda)$ is simple for all $\lambda \in \Lambda_\chi$ by Corollary \ref{P:simplevermas}, so $L_\chi(\lambda) = Z_\chi(\lambda)$. Hence it suffices to show that $\span_{\Fbb_p}\{d\alpha : \alpha \in \Delta\} = \{\lambda \in \Lambda_\chi : \lambda|_{\z(\g)} = 0\}$, but this can be verified by a case-by-case computation on each irreducible component of the root system $\Phi$.

On the other hand, if $\lambda|_{\z(\g)} \neq \mu|_{\z(\g)}$ then $\z(\g_m)$ acts on $L_\chi(\lambda)$ and $L_\chi(\mu)$ by different scalars, so they have different central characters and hence must be non-isomorphic.
\end{proof}

The following Lemma, analogous to a result \cite[10.7]{JaLA} in the reductive case, allows us to extend this classification of simple modules to $p$-characters $\chi = \bkappa_m(e)$ for any $e \in \g$ of standard Levi type.

\begin{Lemma}
\label{L:simplemodulebijection}
There is a bijection between simple $U_\chi(\g_m)$-modules and simple $U_\chi((\g_I)_m)$-modules sending $M$ to $M^{\r_m}$. This bijection takes the simple $U_\chi(\g_m)$-module $L_\chi(\lambda)$ to the simple $U_\chi((\g_I)_m)$-module $L_\chi(\lambda)$.
\end{Lemma}

\begin{proof}
The first part follows from general results on graded modules, namely Corollary 1.4 and Theorems 1.1 and 1.2 in \cite{Sh}. To apply these results, we require only the fact that $\r_m$ and $\r^-_m$ act nilpotently on the baby Verma modules (and hence on their simple quotients) which follows from the definition of $Z_\chi(\lambda)$ and the fact that $\chi(\r^-_m) = 0$. The second part follows easily once we observe that if we take a highest weight generator of weight $\lambda$ for the $U_\chi(\g_m)$ module $L_\chi(\lambda)$, then this element lies in $M^{\r_m}$ and generates it as a $U_\chi((\g_I)_m)$-module.
\end{proof}

\begin{Theorem}
\label{T:simpleclassification}

$L_\chi(\lambda) \cong L_\chi(\mu)$ if and only if $\lambda|_{\z(\g_I)} = \mu|_{\z(\g_I)}$.
\end{Theorem}

\begin{proof}
By Lemma \ref{L:simplemodulebijection}, $L_\chi(\lambda) \cong L_\chi(\mu)$ if and only if the corresponding simple $U_\chi((\g_I)_m)$-modules are isomorphic. But by Proposition \ref{P:regnilpotentsimples} this occurs precisely when $\lambda|_{\z(\g_I)} = \mu|_{\z(\g_I)}$.
\end{proof}

\subsection{The general linear algebra}
\label{ss:generallinear}
Here we explain that our results provide a complete classification of simple $\g_m$-modules with homogeneous $p$-characters when $\g = \gl_n$.

Suppose $\chi \in \g^*$ is homogeneous. Then by Proposition~\ref{P:reductionbydegree}(1) we may assume that $\chi$ is supported in degree $m$. Decomposing $\chi = \chi_s + \chi_n$, the semisimple and nilpotent parts, we consider the centraliser $(\g_m)^{\chi_s} = (\g^{\chi_s})_m$. Note that $\g^{\chi_s}$ is a Levi subalgebra of $\gl_n$, and therefore it is isomorphic to $\l = \gl_{n_1} \times \cdots \times \gl_{n_k}$ for some $\sum_{i=1}^k n_i$ and $1\le k \le n$. Furthermore $\chi|_{\l_m}$ is supported on the centre.

The classification of simple $U_\chi(\l)$-modules follows easily from the classification of the simple $U_{\chi|_{\gl_{n_i}}}(\gl_{n_i})$ modules for $i=1,...,k$ and so we have now reduced the problem to classifying simple $U_\chi(\g_m)$-modules, where the semisimple part of $\chi$ is supported on the centre of $\g_m$. Using Lemma~\ref{L:reducingoverthecentre} we may assume that $\chi$ is nilpotent and supported in degree $m$.

Every nilpotent $\psi \in \gl_n$ can be placed in standard Levi form and now the classification of simple $U_\chi(\gl_n)$-modules can be deduced from the results of Section~\ref{ss:simpleclassification}.

\begin{Remark}
This classification can even be applied in the case where the semisimple and nilpotent parts are each homogeneous, supported in respective degrees $i, j$ with $i \ge j$.
\end{Remark}

\begin{Remark}
Outside type {\sf A} not all nilpotent elements have standard Levi type and so the above classification breaks down. If $p > 2$ and $\g$ is simple of type {\sf B, C} or {\sf D} then the nilpotent orbits are still classified by partitions (with an extra decoration in type {\sf D}). It is not hard to see that the partitions corresponding to standard Levi nilpotent orbits are as follows:
\begin{itemize}
\item[(type {\sf B})] All parts of the partition occur with even multiplicity, with the possible exception of one odd part, which can occur with odd multiplicity.
\item[(type {\sf C})] All parts occur with even multiplicity, with the possible exception of one even part, which can occur with odd multiplicity.
\item[(type {\sf D})] For $\so_{2n}$ there are $2k$ parts which occur with even multiplicity, say $\lambda_{j_1},...,\lambda_{j_{2k}}$. The final two parts of the partition have sizes $2n - \sum_{i=1}^{2k} \lambda_{j_i} -1$ and $1$.
\end{itemize}
When the nilpotent part of a homogeneous $p$-character corresponds to a partition satisfying these properties we obtain a classification of simple $U_\chi(\g)$-modules from the above results in the same manner as in the type A case.
\end{Remark}

\section{Cartan invariants for the restricted enveloping algebra}
\label{S:Cartaninv}

Recall that for any finite dimensional algebra $A$ with simple modules $L_1,...,L_s$ there are projective covers $P_1,...,P_s$, unique up to isomorphism. The composition multiplicities $[P_i : L_j]$ are known as the Cartan invariants of $A$.

In Corollary~\ref{C:semisimplecharactermodules} and Proposition~\ref{P:torusprojectives} we effectively calculated the Cartan invariants of $U_\chi(\g_m)$ whenever $\chi$ is regular semisimple. In this last section of the paper we deal with the opposite extreme: $\chi = 0$.

\subsection{Restricted Verma modules of two flavours}

Recall from earlier that we have a triangular decomposition $\g = \n^- \oplus \h \oplus \n^+$, where $\h = \Lie(T)$ for some choice of torus $T \subseteq G_m$. We consider baby Verma modules $Z_0(\lambda)$ for the restricted enveloping algebra $U_0(\g_m)$. Since $\chi = 0$ is fixed, from now on we omit the subscript in the notation for baby Verma modules and instead write $Z(\lambda)$. We also define $Z^{\g}(\lambda)$ to be the baby Verma module for $U_0(\g)$:
\[Z^{\g}(\lambda) := U_0(\g) \otimes_{U_0(\b^+)} \k_\lambda\]
where $\k_\lambda$ is the 1-dimensional $\b^+$ module on which $\h$ acts by $\lambda$ and $\n^+$ acts by 0.

We can inflate this to a $U_0(\g_m)$-module by letting $\g_m^{(\geq 1)}$ act by 0; by abuse of notation we also label this module $Z^{\g}(\lambda)$. Now observe that by Proposition \ref{P:reductionbydegree}, we have that the simple $U_0(\g_m)$-modules are just the simple modules for $U_0(\g)$ with $\g_m^{(\geq 1)}$ acting by 0. As with the baby Verma modules, these modules are labelled by $\Lambda_0$, so again abusing notation we write $L(\lambda)$ for the both simple $U_0(\g)$-module of weight $\lambda$ and the simple $U_0(\g_m)$-module of weight $\lambda$. (Recall from the previous section that we may view $\Lambda_0$ as a subset of $\h^*$)

We make the following notation: let $N$ be a $U_0(\g_m)$-module. Then we write $[N:L(\lambda)]$ for the composition multiplicity of $L(\lambda)$ in $N$. Similarly, if $\{M(\lambda): \lambda \in \Lambda_0\}$ is a family of $U_0(\g_m)$-modules, and $N$ admits a filtration such that each section is isomorphic to some $M(\lambda)$ then we write $(N:M(\lambda))$ to denote the number of times $M(\lambda)$ occurs as a subquotient in the filtration. Of course, the quantities $(N:M(\lambda))$ may not be well-defined in general, however we shall show that they are whenever we use this notation.

The following proposition is not difficult to prove and is useful in the computation of composition multiplicities.

\begin{Proposition}
\label{P:compmultsums}
If $(N: M(\mu))$ exists, then for any $\lambda \in \Lambda_0$ we have $[N:L(\lambda)] = \sum_{\mu \in \Lambda_0} (N: M(\mu))[M(\mu): L(\lambda)]$.
\end{Proposition}

\begin{proof}
Let $0 = N_0 \subset N_1 \subset \dots \subset N_k = N$ be a filtration of $N$ in which for all $\mu \in \Lambda_0$, $M(\mu)$ appears as a quotient $N_i/N_{i-1}$ precisely $(N: M(\mu))$ times. Then refining this filtration to a composition series we see that $[N:L(\lambda)] = \sum_{i = 1}^{k} [N_i/N_{i-1}:L(\lambda)]$ and the desired result follows.
\end{proof}

\subsection{Graded $U_0(\g_m)$-modules}

We now wish to compute the composition multiplicities of projective $U_0(\g_m)$-modules. One tool we will use in doing so is a grading on $U_0(\g_m)$ and a category of graded $U_0(\g_m)$-modules, generalising the well-known technique in the $m = 0$ case (see \cite[\textsection 11]{JaLA} for example).

We grade $U_0(\g_m)$ by $X^*(T)$, the character lattice of $T$, in the following way. First define a $X^*(T)$ grading on $U(\g_m)$ by letting $\g_\alpha t^i$ lie in grading $\alpha$, $\h_m$ lie in grading 0, and extending to all of $U(\g_m)$. Then observe that the kernel of the quotient map $U(\g_m) \rightarrow U_0(\g_m)$ is generated by homogeneous elements, so this grading decends to an $X^*(T)$ grading on $U_0(\g_m)$. We consider certain $X^*(T)$-graded $U_0(\g_m)$-modules whose gradings are compatible with the action of action of $\h$ in the following sense. Let $M = \bigoplus_{\gamma \in X^*(T)} M_\gamma$ be an $X^*(T)$-graded $U_0(\g_m)$-module. We consider the category $\mathcal{C}$ of such modules where $\h$ acts on $M_\gamma$ by $d\gamma$ for all $\gamma \in X^*(T)$. The observations of \cite[\textsection 11.4]{JaLA} actually show that this special case allows one to understand the category of all graded modules.

For $M \in \mathcal{C}$, we write $\Char M$, the character of $M$, for the function $X^*(T) \rightarrow \Z_{\geq 0}$ which sends $\gamma \in X^*(T)$ to $\dim M_\gamma$. We use the same notation for composition multiplicities in $\mathcal{C}$ as in the ungraded case. We emphasise that graded and ungraded composition multiplicities are not the same, since a shift by $p \gamma$ in the grading of any module in $\mathcal{C}$ subtends a non-isomorphic graded module with isomorphic underlying ungraded module.

Let $\gamma \in X^*(T)$. Then we define the {\it graded baby Verma module} $\widehat{Z}(\gamma)$ to be the module $Z(d\gamma)$ with grading given by letting $\prod (f_\alpha t^i)^{m_{\alpha, i}} \otimes 1_{d\gamma}$ lie in grading $\gamma - \sum \alpha m_{\alpha, i}$. As noted above, if $\beta \in X^*(T)$, then $\widehat{Z}(\gamma)$ and $\widehat{Z}(\gamma + p \beta)$ have isomorphic module structures but different gradings. We also define $\widehat{Z}^\g(\gamma)$ in a similar way.

\begin{Lemma}
For all $\gamma \in X^*(T)$, $\widehat{Z}(\gamma) \in \mathcal{C}$ and $\widehat{Z}^\g(\gamma) \in \mathcal{C}$.
\end{Lemma}

\begin{proof}
The first claim follows immediately from the observation in $\widehat{Z}(\gamma)$ that $\h_m^{(0)}$ acts on $\prod (f_\alpha t^i)^{m_{\alpha, i}} \otimes 1_{d\gamma}$ by $d\gamma - \sum (d\alpha) m_{\alpha, i}$, and a similar argument applies for $\widehat{Z}^\g(\gamma)$.
\end{proof}

By Lemma \ref{L:babyvermas2} and standard results on graded modules, we also see that $\widehat{Z}(\gamma)$ has a unique simple quotient $\widehat{L}(\gamma)$ which is isomorphic to $L(d\gamma)$ as a $U_0(\g_m)$-module.

\begin{Proposition}
The $\widehat{L}(\gamma)$ form a complete set of isomorphism classes of simple objects in $\mathcal{C}$, and the $\Char \widehat{L}(\gamma)$ form a basis for the additive group of linear functions $X^*(T) \rightarrow \Z$, and hence the graded composition multiplicities of any $M \in \mathcal{C}$ are determined by $\Char M$.
\end{Proposition}

\begin{proof}
By \cite[\S1.5]{JaMR}, whose argument applies here, all simple objects in $\mathcal{C}$ are simple as $U_0(\g_m)$-modules and by Theorem \ref{T:simpleclassification} we have a classification of simple $U_0(\g_m)$-modules. Let $M \in \mathcal{C}$ be a simple object isomorphic as a $U_0(\g_m)$-module to $L(\lambda)$ for some $\lambda \in \Lambda_0$. For any grading on $L(\lambda)$, the element $1 \otimes 1_\lambda$ must be homogeneous and so must lie in grading $\gamma$ for some $\gamma \in X^*(T)$ such that $d\gamma = \lambda$. Then by considering the grading on $U_0(\g_m)$ and module structure of $L(\lambda)$, we must have $M \cong \widehat{L}(\gamma)$. The second assertion follows from the observation that $\Char(\widehat{L}(\gamma))(\gamma) = 1$ and $\Char(\widehat{L}(\gamma))(\beta) = 0$ unless $\beta \leq \gamma$.
\end{proof}

\subsection{Composition multiplicities of restricted baby Verma modules}

We now use the graded versions of the simple modules just introduced to compute the composition multiplicities of the baby Verma modules, which is a vital intermediate step towards our goal of computing the Cartan invariants for $U_0(\g_m)$. The following formula allows us to express these in terms of the composition multiplicities for the baby Verma modules for the original reductive Lie algebra $\g$.

\begin{Theorem}
\label{T:babyvermacompmults}
For any $\lambda, \mu \in \Lambda_0$, we have:
\[[Z(\lambda): L(\mu)] = \begin{cases}
l_\mu p^{\frac{m}{2} (\dim(\g) - \rank(\g)) - \rank(\g)} \mbox{ if } \lambda|_{\z(\g)} = \mu|_{\z(\g)}\\
0 \mbox{ otherwise}
\end{cases}\]
where $l_\mu = \sum_{\nu \in \Lambda_0} [Z^{\g}(\nu): L(\mu)]$. In particular, these composition multiplicities depend only on $\lambda|_{\z(\g)}$.
\end{Theorem}

\begin{proof}

We will first prove a relationship between the characters of the families of graded modules $\widehat{Z}(\gamma)$ and $\widehat{Z}^\g(\gamma)$. This will then allows us to deduce a relationship between their composition multiplicities in the category $\mathcal{C}$ of graded modules introduced in the previous section, from which the theorem follows by forgetting the gradings everywhere.

Start by defining a generalisation of Kostant's partition function $p_m : X^*(T) \rightarrow \Z_{\geq 0}$ by
\[p_m(\gamma) = |\{(m_{i, \alpha})_{\alpha \in \Phi^+, 1 \leq i \leq m} : \sum \alpha m_{i, \alpha} = \gamma, 0 \leq m_{i, \alpha} \leq p-1\}|\]
and observe in particular that $p_m(\gamma) = 0$ if $(d\gamma)|_{\z(\g)} \neq 0$. Let $I = \{((m_{i, \alpha})_{\alpha \in \Phi^+, 1 \leq i \leq m} : 0 \leq m_{i, \alpha} \leq p-1\}$ and fix $\gamma \in X^*(T)$. For each $\mathbf{m} \in I$, we can define an $X^*(T)$-graded subspace $\widehat{Z}_{\mathbf{m}} = \{\prod f_\alpha^{n_\alpha}\prod(f_\alpha t^i)^{m_{\alpha, i}} \otimes 1_{\gamma} : 0 \leq n_\alpha \leq p-1\}$ of $\widehat{Z}(\gamma)$, and furthermore we have a decomposition $\widehat{Z}(\gamma) = \bigoplus_{\mathbf{m} \in I} \widehat{Z}_{\mathbf{m}}$. Now observe that $\Char \widehat{Z}_{\mathbf{m}} = \Char \widehat{Z}^{\g}(\gamma - \sum \alpha m_{i, \alpha})$. Hence we obtain the formula $\Char \widehat{Z}(\gamma) = \sum_{\beta \in X^*(T)} p_m(\gamma - \beta) \Char \widehat{Z}^\g(\beta)$, and can then immediately deduce that $[\widehat{Z}(\gamma):\widehat{L}(\delta)] = \sum_{\beta \in X^*(T)} p_m(\gamma - \beta) [\widehat{Z}^\g(\beta) : \widehat{L}(\delta)]$.

Forgetting the gradings on these modules, we see that $[Z(d\gamma):L(d\delta)] = \sum_{\beta \in X^*(T)} p_m(\gamma - \beta) [Z^\g(d\beta) : L(d\delta)]$. Set $\lambda = d\gamma$, $\mu = d\delta$, and $\nu = d\beta$. If $\lambda|_{\z(\g)} \neq \mu|_{\z(\g)}$ then for all $\beta \in X^*(T)$ either $p_m(\gamma - \beta) = 0$ or $[Z^\g(\nu) : L(\mu)] = 0$ so $[Z(\lambda) : L(\mu)] = 0$. On the other hand, if $\lambda|_{\z(\g)} = \mu|_{\z(\g)}$, then using the fact that $d\beta = d(\beta + p\xi)$ for any $\xi \in X^*(T)$ we have $[Z(\lambda): L(\mu)] = \sum_{\delta \in X^*(T)} p_m(\gamma - p \delta) \sum_{\nu \in \Lambda_0} [Z^{\g}(\nu): L(\mu)]$. Hence it remains only to show that for any $\gamma \in X^*(T)$, $\sum_{\delta \in X^*(T)} p_m(\gamma - p \delta) = p^{\frac{m}{2}(\dim(\g) - \rank(\g)) - \rank(\g)}$.

Let $S_{\gamma} = \bigcup_{\delta \in X^*(T)}\{(m_{i, \alpha})_{\alpha \in \Phi^+, 1 \leq i \leq m} : \sum \alpha m_{i, \alpha} = \gamma - p\delta, 0 \leq m_{i, \alpha} \leq p-1\}$. Then observe that for any $\beta \in \Phi^+$ there is a bijection $f: S_{\gamma} \rightarrow S_{\gamma + \beta}$ given by
\[f(\mathbf{m})_{i, \alpha}  = \begin{cases}
m_{i, \alpha} & \mbox{ if } i > 0 \mbox{ or } \alpha \neq \beta\\
m_{i, \alpha} + 1 \mbox{ mod }p & \mbox{ if } i = 0 \mbox{ and } \alpha = \beta
\end{cases}\]

Hence $|S_{\gamma}| = |S_{\delta}|$ for all $\gamma, \delta \in X^*(T)$. Now, observe that $\sum_{\gamma \in X^*(T)} p_m(\gamma) = p^{\frac{m}{2}(\dim(\g) - \rank(\g))}$ and we have $p^{\rank(\g)}$ distinct sets $S_{\gamma}$, so for any $\gamma \in X^*(T)$, $|S_{\gamma}| = \sum_{\delta \in X^*(T)} p_m(\gamma - p \delta) = p^{\frac{m}{2}(\dim(\g) - \rank(\g)) - \rank(\g)}$, completing the proof.
\end{proof}

\begin{Corollary}
\label{C:compfactorsindepend}
If $\z(\g) = 0$ then the composition multiplicities $[Z(\lambda) : M(\lambda)]$ are independent of $\lambda$, so all baby Verma modules have the same composition factors and multiplicities.
\end{Corollary}

The fact the composition multiplicities depend only on $\lambda|_{\z(\g)}$ is analogous to a result that states that for a the Lie algebra of a reductive group, restricted baby Verma modules with linked weights all have the same composition factors. The difference in this case is that the only condition for our weights to be linked is that they take the same values on $\z(\g)$. This bears some stark similarity to the previous work of the authors on truncated current Lie algebras in characteristic zero, see \cite[Theorem~4.2]{CT} for example.

\begin{Corollary}
\label{C:blocks}
Let $\lambda, \mu \in \Lambda_0$. If $\lambda|_{\z(\g)} = \mu|_{\z(\g)}$, then $L(\lambda)$ and $L(\mu)$ lie in the same block of $U_0(\g_m)$. In particular, if $\z(\g) = 0$ then $U_0(\g_m)$ has only one block.
\end{Corollary}

\begin{proof}
Suppose $\lambda|_{\z(\g)} = \mu|_{\z(\g)}$. Then by Corollary~\ref{C:compfactorsindepend} we have that $[Z(\lambda): L(\lambda)]$ and $[Z(\lambda): L(\mu)]$ are both non-zero, and by Lemma \ref{L:babyvermas2} we have that $Z(\lambda)$ is indecomposable. Hence $L(\lambda)$ and $L(\mu)$ lie in the same block.
\end{proof}

\begin{Remark}
\label{R:blocks}
\begin{enumerate}
\item Theorem \ref{T:restrictedprojectivescompmults} will imply that the condition $\lambda|_{\z(\g)} = \mu|_{\z(\g)}$ is also necessary for $L(\lambda)$ and $L(\mu)$ to lie in the same block.
\item More generally it is easy to see that the number of blocks is $p^{\dim \z(\g_m)} = p^{(m+1) \dim \z(\g)}$.
\end{enumerate}

\end{Remark}

\subsection{Cartan invariants for $U_0(\g_m)$}

Recall the notation $Q^{\h_m}(\lambda)$ from Proposition~\ref{P:torusprojectives}. We define the following families of $U_0(\g_m)$-modules:
\[Z_{proj}(\lambda) = U_0(\g_m) \otimes_{U_0(\b^+_m)} Q^{\h_m}(\lambda)\]
\[DZ(\lambda)  = (U_0(\g_m) \otimes_{U_0(\b^-_m)} (\k_\lambda)^*)^*\]
where for $Z_{proj}(\lambda)$ we inflate $Q^{\h_m}(\lambda)$ to a $U_0(\b^+_m)$-module by letting $\n^+_m$ act by 0, and for $DZ(\lambda)$ we inflate $(\k_\lambda)^*$ to a $U_0(\b^-_m)$-module by letting $\n^-_m$ act by 0. Here the module structure on the dual $M^*$ of a $U_0(\l)$-module $M$ for some Lie algebra $\l$ is given by $(x \cdot f)(m) = -f(x \cdot m)$ for $x \in \l$, $f \in M^*$, and $m \in M$. Finally, we write $Q(\lambda)$ for the projective cover of $L(\lambda)$.

These definitions allow us to state the following theorem, which is a special case of \cite[Theorem 1.3.6]{Na}.

\begin{Theorem}
\label{T:reciprocity}
For all $\lambda, \mu \in \Lambda_0$, we have $(Q(\lambda) : Z_{\proj}(\mu)) = [DZ(\mu): L(\lambda)]$. $\hfill\qed$
\end{Theorem}

This theorem is the key result which allows us to prove the following result giving a formula for the Cartan invariants of $U_0(\g_m)$.

\begin{Theorem}
\label{T:restrictedprojectivescompmults}
\[[Q(\lambda): L(\mu)] = \begin{cases}
l_\lambda l_\mu p^{m \dim(\g) - \rank(\g)} \mbox{ if } \lambda|_{\z(\g)} = \mu|_{\z(\g)}\\
0 \mbox{ otherwise} \end{cases}\]
where again $l_\lambda = \sum_{\nu \in \Lambda_0} [Z^{\g}(\nu):L(\lambda)]$.
\end{Theorem}

\begin{proof}
We aim to find formulae for $[DZ(\mu) : L(\lambda)]$ and $(Z_{\proj}(\mu) : Z(\lambda))$; the result will then follow by Theorem \ref{T:reciprocity} and several applications of Proposition \ref{P:compmultsums}. First, we claim that given $\gamma \in X^*(T)$ such that $\mu = d\gamma$ we can construct an $X^*(T)$-grading on $DZ(\mu)$ such that the graded module $\widehat{DZ}(\gamma) \in \mathcal{C}$ has the same character as $\widehat{Z}(\gamma)$; it is then immediate that $[DZ(\mu): L(\lambda)] = [Z(\mu):L(\lambda)]$ for all $\lambda, \mu \in \Lambda_0$. Observe that $(\k_\mu)^* \cong \k_{-\mu}$, so we can define an $X^*(T)$-grading on $U_0(\g_m) \otimes_{U_0(\b^-_m)} (\k_\mu)^*$ by letting $\prod (e_\alpha t^i)^{m_{\alpha, i}} \otimes 1_{-\mu}$ lie in grading $-\gamma + \sum \alpha m_{\alpha, i}$.

Now, for any finite-dimensional $M \in \mathcal{C}$ we can choose a basis $\{x_1, \dots, x_k\}$ such that $x_i$ lies in grading $\delta_i$, and then define a grading on the dual module $M^*$ by choosing a dual basis $\{f_1, \dots, f_k\}$ and letting $f_i$ lie in grading $-\delta_i$. The resulting graded module $M^*$ also lies in $\mathcal{C}$, and we see that $\Char M^* (\delta) = \Char M (-\delta)$. Applying this with $M = U_0(\g_m) \otimes_{U_0(\b^-_m)} (\k_\mu)^*$ equipped with the grading described earlier, we obtain a graded module $\widehat{DZ}(\gamma)$ such that $\Char \widehat{DZ}(\gamma) = \Char \widehat{Z}(\gamma)$ as required.

Next, we observe that $U_0(\g_m)$ is free as a $U_0(\b^+_m)$-module and hence the functor $U_0(\g_m) \otimes_{U_0(\b^+_m)} (-)$ is exact. Hence by Proposition \ref{P:torusprojectives} we see that $(Z_{proj}(\lambda):Z(\mu)) = \delta_{\lambda, \mu} p^{m \rank(\g)}$. We can then apply Proposition \ref{P:compmultsums} to compute:
\begin{eqnarray*}
[Q(\lambda):L(\mu)] &=& \sum_{\nu_1 \in \Lambda_0} (Q(\lambda): Z_{proj}(\nu_1)) [Z_{proj}(\nu_1):L(\mu)] \\
&=& \sum_{\nu_1, \nu_2 \in \Lambda_0} (Q(\lambda): Z_{proj}(\nu_1))(Z_{proj}(\nu_1) : Z(\nu_2))[Z(\nu_2):L(\mu)]\\
&=& \sum_{\nu_1, \nu_2 \in \Lambda_0} [Z(\nu_1): L(\lambda)](Z_{proj}(\nu_1) : Z(\nu_2))[Z(\nu_2):L(\mu)]\\
&=& \sum_{\nu \in \Lambda_0} [Z(\nu): L(\lambda)] p^{m \rank(\g)}[Z(\nu):L(\mu)]\\
&=&\begin{cases}
 l_\lambda l_\mu p^{m \dim(\g) - \rank(\g)} \mbox{ if } \lambda|_{\z(\g)} = \mu|_{\z(\g)}\\
0 \mbox{ otherwise}
\end{cases}
\end{eqnarray*}
\end{proof}

\begin{Remark}
\label{R:final}
It is possible to the use results of \cite{RW} along with the theory of translation functors to deduce that for $p > 2h-1$ these Cartan invariants can actually be calculated in terms of the machinery of $p$-Kazhdan--Lusztig polynomials, which arise from Elias--Williamson's Hecke category (see {\it op. cit.} for further references).
\end{Remark}

\vspace{10pt}

\noindent {Contact details:}\vspace{5pt}

Matthew Chaffe {\sf mxc167@student.bham.ac.uk};\\
{School of Mathematics, University of Birmingham, Edgbaston, Birmingham, B15 2TT, UK.}\vspace{5pt}

Lewis Topley {\sf lt803$@$bath.ac.uk};\\
{Department of Mathematical Sciences, University of Bath, Claverton Down, Bath, BA2 7AY, UK.}

\end{document}